\documentclass[12pt]{article}
\usepackage{amsmath}
\usepackage{amssymb}
\usepackage{amsthm}
\usepackage{extarrows}

\newtheorem{theorem}{Theorem}
\newtheorem*{theorem*}{Theorem}

\newtheorem{lemma}{Lemma}

\newtheorem{proposition}{Proposition}
\newtheorem*{proposition*}{Proposition}

\newcommand{\Mk}{\mathcal{M}_k}
\newcommand{\Mkk}{\mathcal{M}_{k-1}}
\newcommand{\Hk}{\mathcal{H}_k}
\newcommand{\Rm}{\mathbb{R}^m}
\newcommand{\C}{\mathbb{C}}
\newcommand{\Clm}{\mathcal{C}l_m}

\newcommand{\scs}{\mathcal{S}}

\newcommand{\Sm}{\mathbb{S}^{m-1}}

\newcommand{\udx}{\langle u,D_x\rangle}
\newcommand{\dudx}{\langle D_u,D_x\rangle}
\newcommand{\fone}{\mathfrak{f}_1}
\newcommand{\Dtwo}{\mathcal{D}_2}

\newcommand{\Deven}{\mathcal{D}_{2j}}
\newcommand{\Dodd}{\mathcal{D}_{2j-1}}

\begin{document}

\title{Construction of Arbitrary Order Conformally Invariant Operators in Higher Spin Spaces}

\author{Chao Ding$^1$\thanks{Electronic address:  {\tt dchao@uark.edu}.}, Raymond Walter$^{1,2}$\thanks{Electronic address:  {\tt rwalter@email.uark.edu}; R.W. acknowledges this material is based upon work supported by the National Science Foundation Graduate Research Fellowship Program under Grant No. DGE-0957325 and the University of Arkansas Graduate School Distinguished Doctoral Fellowship in Mathematics and Physics.}, and John Ryan$^1$\thanks{Electronic address: {\tt jryan@uark.edu.}} \\
\emph{\small $^1$Department of Mathematics, University of Arkansas, Fayetteville, AR 72701, USA} \\ 
\emph{\small $^2$Department of Physics, University of Arkansas, Fayetteville, AR 72701, USA}}
\date{}

\maketitle

\begin{abstract}
This paper completes the construction of arbitrary order conformally invariant differential operators in higher spin spaces. Jan Slov\'{a}k has classified all conformally invariant differential operators on locally conformally flat manifolds. We complete his results in higher spin theory in Euclidean space by giving explicit expressions for arbitrary order conformally invariant differential operators, where by conformally invariant we mean equivariant with respect to the conformal group of $S^m$ acting in Euclidean space $\mathbb{R}^m$. We name these the fermionic operators when the order is odd and the bosonic operators when the order is even. Our approach explicitly uses convolution type operators to construct conformally invariant differential operators. These convolution type operators are examples of Knapp-Stein operators and they can be considered as the inverses of the corresponding differential operators. Intertwining operators of these convolution type operators are provided and intertwining operators of differential operators follow immediately. This reveals that our convolution type operators and differential operators are all conformally invariant. This also gives us a class of conformally invariant convolution type operators in higher spin spaces. Their inverses, when they exist, are conformally invariant pseudo-differential operators. Further we use Stein Weiss gradient operators and representation theory for the Spin group to naturally motivate the construction of Rarita-Schwinger operators.
\end{abstract}
{\bf Keywords:}\quad Fermionic operators, Bosonic operators, Conformal invariance, Fundamental solutions, Intertwining operators, Convolution type operators, Knapp-Stein operators.\\
{\bf AMS subject classification:}\quad Primary 53A30, secondary 20G05, 30G35
%%%%%%%%%%%%% Introduction %%%%%%%%%%%%%%
\section{Introduction}\hspace*{\fill} 
\par
The \emph{higher spin theory} in Clifford analysis was first introduced with the Rarita-Schwinger operators \cite{B}. This theory considers generalizations of classical Clifford analysis techniques to higher spin spaces \cite{B1, Br1, B, D, E, Li}, focusing on operators acting on functions on $\Rm$ that take values in arbitrary irreducible representations of $Spin(m)$. Generally these are polynomial representations, such as spaces of $k$-homogeneous monogenic or harmonic polynomials ($\Mk$ or $\Hk$) corresponding to particles of half-integer spin or integer spin. Here monogenic functions are solutions to the Euclidean Dirac equation. \\
\par
After the Laplacian was pointed out no longer to be conformally invariant in higher spin space, Eelbode and Roels \cite{E} constructed a second order conformally invariant operator: the (generalized) Maxwell operator acting on $C^{\infty}(\Rm,\mathcal{H}_1)$, where the target space $\mathcal{H}_1$ is a degree-$1$ homogeneous harmonic polynomial space. De Bie and his co-authors \cite{B1} then generalized this Maxwell operator to the case when it acts on $C^{\infty}(\Rm,\mathcal{H}_k)$. This is what they call the higher spin Laplace operator. In \cite{Ding1}, we introduced fermionic operators and bosonic operators as the generalization of $k$-th powers of the Euclidean Dirac operator to higher spin space, considering the special case of target space of degree-$1$ homogeneous polynomials while using similar techniques as in \cite{B1,E}. The connections to mathematical physics emphasized in that work also apply to the present manuscript. We later constructed the 3rd order fermionic and 4th order bosonic operators when the target space is a degree-$k$ homogeneous polynomial space \cite{Ding}. Unfortunately, the generalized symmetry approach we used in \cite{Ding,Ding1} was computationally infeasible for arbitrary higher order conformally invariant operators. \\
\par
The methods we use to construct conformally invariant operators are usually either of the following type. 
\begin{enumerate}
\item Verify some differential operator is conformally invariant under M\"{o}bius transformations with the help of the Iwasawa decomposition of the M\"{o}bius transformation, for instance as in \cite{D}.
\item Show the generalized symmetries of some differential operator generate a conformal Lie algebra, for instance as in \cite{B1,E}. 
\end{enumerate}
\par
This paper uses a method different from these. We start by applying Slov\'{a}k \cite{J} and Sou\v{c}ek's \cite{Vlad} results with arguments of Bure\v{s} et al. \cite{B} to get fundamental solutions of arbitrary order conformally invariant differential operators in higher spin spaces. Then we only need to  construct differential operators with those specific fundamental solutions. In particular, from the fundamental solutions of first and second order conformally invariant differential operators obtained from the preceding argument, we can also find the Rarita-Schwinger operators \cite{B} and higher spin Laplace operators \cite{B1} by verifying they have such fundamental solutions. Arguing by induction, we then complete the work on constructing conformally invariant operators in higher spin spaces by providing explicit forms of arbitrary $j$-th order conformally invariant operators in higher spin spaces with $j>2$.\\
\par
Notably, we discover a new analytic approach to show that a differential operator is conformally invariant. More specifically, we use its fundamental solution to define a convolution type operator and then the fundamental solution can be realized as the inverse of the corresponding differential operator in the sense of such convolution. Hence, if we can show the fundamental solution (as a convolution operator) is conformally invariant, then as the inverse, the corresponding differential operator will also be conformally invariant. Thus the intertwining operators of the fundamental solution (as a convolution operator) are the inverses of the intertwining operators of the differential operators. This idea brings us an infinite class of conformally invariant convolution type operators in higher spin spaces; their inverses, when they exist, are generalized conformally invariant pseudo-differential operators. More details can be found in Section $4.1$. It is worth pointing out that these intertwining operators and convolution type operators are special cases of Knapp-Stein intertwining operators and Knapp-Stein operators in higher spin theory (\cite{CO,KS}).\\
\par
Our study of conformally invariant differential operators in higher spin spaces suggests a distinct Representation-Theoretic approach to Clifford analysis, in contrast to the classical Stokes approach. In the latter approach, the motivation for Dirac-type operators is to obtain operators satisfying a Stokes-type theorem. This does not need irreducible representation theory. In contrast, in the Representation-Theoretic approach, we consider functions taking values in irreducible representations of the Spin group. This forces one to consider irreducible representation theory, as happens elsewhere in the literature where Dirac operators are used \cite{G} and especially in spin geometry \cite{Lawson}. Moreover, irreducible spin representations are natural for studying spin invariance and in particular conformal invariance. That is not to dismiss the Stokes approach\textemdash it is used, for instance, to establish the $L^2$ boundedness of the double layer potential operator on Lipschitz graphs \cite{Mitrea}, and other applications are found in such works as \cite{Belgians}. Though the present work aims to demonstrate the value of the Representation-Theoretic approach, in future work the two distinct approaches may complement each other.\\
\par
The paper is organized as follows. We briefly introduce Clifford algebras, Clifford analysis, and representation theory of the Spin group in Section 2. We recall the Stein-Weiss construction of the Euclidean Dirac operator and Rarita-Schwinger operator from \cite{Ding0,SR} in Section 3. This motivates the extensive use of representation theory in our recent work on conformally invariant differential operators in higher spin theory. Further, this construction also reveals that Stein-Weiss gradient operators and representation theory of the Spin group provide the most natural approach to the study of Rarita-Schwinger operators.\\
\par
In Section 4, we provide an approach different from \cite{B1,E} to construct these conformally invariant differential operators. This approach relies heavily on the fundamental solutions of these conformally invariant differential operators. We also define a convolution type operator associated to each fundamental solution to show each fundamental solution is actually the inverse of the corresponding differential operator. An explicit proof for the intertwining operators of these convolution type operators is provided there. This implies conformal invariance of these convolution type operators and conformal invariance of the corresponding differential operators is shown immediately. We point out that this idea also gives an infinite class of conformally invariant convolution type operators; their inverses, when they exist, are generalized conformally invariant pseudo-differential operators. We also show that the higher spin Laplace operators \cite{B1} can also be derived from this approach. Then we introduce bosonic operators $\Deven$ as the generalization of $D_x^{2j}$ when acting on $C^{\infty}(\Rm, \Hk)$ and fermionic operators $\Dodd$ as the generalization of $D_x^{2j-1}$ when acting on $C^{\infty}(\Rm, \Mk)$, where $D_x$ is the Euclidean Dirac operator with respect to the variable $x$. The connections between these and lower order conformally invariant operators are also revealed in the construction. Moreover, since the construction is explicitly based on the uniqueness of the operators and their fundamental solutions with the appropriate intertwining operators for a conformal transformation, the conformal invariance and fundamental solutions of the bosonic and fermionic operators arise naturally in our formalism. 
\par
We cover technical details and proofs for the fermionic case in Section 5.

\section*{Acknowledgement}
The authors are grateful to Bent \O rsted for communications pointing out that the intertwining operators for our fermionic (bosonic) operators and the convolution type operators defined with the fundamental solutions of fermionic (bosonic) operators are special cases of Knapp-Stein intertwining operators and Knapp-Stein operators in higher spin theory (\cite{CO,KS}). The authors are also grateful to the referee for helpful comments.
%%%%%%%%%%%         Preliminaries       %%%%%%%%%%%%%%%%%%%
\section{Preliminaries}
\subsection{Clifford algebra}\hspace*{\fill}
A real Clifford algebra, $\mathcal{C}l_m,$ can be generated from the $m$-dimensional real Euclidean space $\mathbb{R}^m$ by considering the
relationship $$\underline{x}^{2}=-\|\underline{x}\|^{2}$$ for each
$\underline{x}\in \mathbb{R}^m$.  We have $\mathbb{R}^m\subseteq \mathcal{C}l_m$. If $\{e_1,\ldots, e_m\}$ is an orthonormal basis for $\mathbb{R}^m$, then $\underline{x}^{2}=-\|\underline{x}\|^{2}$ tells us that $$e_i e_j + e_j e_i= -2\delta_{ij},$$ where $\delta_{ij}$ is the Kronecker delta function. An arbitrary element of the basis of the Clifford algebra can be written as $e_A=e_{j_1}\cdots e_{j_r},$ where $A=\{j_1, \cdots, j_r\}\subset \{1, 2, \cdots, m\}$ and $1\leq j_1< j_2 < \cdots < j_r \leq m.$
Hence for any element $a\in \mathcal{C}l_m$, we have $a=\sum_Aa_Ae_A,$ where $a_A\in \mathbb{R}$. Similarly, the complex Clifford algebra $\Clm (\C)$ is defined as the complexification of the real Clifford algebra
$$\Clm (\C)=\Clm\otimes\C.$$
We consider the real Clifford algebra $\Clm$ throughout this subsection, but in the rest of the paper we consider the complex Clifford algebra $\Clm (\mathbb{C})$ unless otherwise specified. The complex Clifford algebra may be viewed as a vector space over the field of scalars $\mathbb{C}$; correspondingly, in this work we may refer to complex-valued functions as scalar-valued functions. Alternatively, there is an isomorphic copy of $\mathbb{C}$ embedded in $\Clm (\mathbb{C})$ that may be considered as a scalar subspace. 

For $a=\sum_Aa_Ae_A\in\Clm$, we define the reversion of $a$ as
\begin{eqnarray*}
\tilde{a}=\sum_{A}(-1)^{|A|(|A|-1)/2}a_Ae_A,
\end{eqnarray*}
where $|A|$ is the cardinality of $A$. In particular, $\widetilde{e_{j_1}\cdots e_{j_r}}=e_{j_r}\cdots e_{j_1}$. Also $\tilde{ab}=\tilde{b}\tilde{a}$ for $a, b\in\Clm.$
\\
\par
The Pin and Spin groups play an important role in Clifford analysis. The Pin group can be defined as $$Pin(m)=\{a\in \mathcal{C}l_m: a=y_1y_2\dots y_{p},\ y_1,\dots,y_{p}\in\mathbb{S}^{m-1},p\in\mathbb{N}\},$$ 
where $\mathbb{S} ^{m-1}$ is the unit sphere in $\Rm$. $Pin(m)$ is clearly a multiplicative group in $\mathcal{C}l_m$, see \cite{Belgians} for more details. \\
\par
Now suppose $a\in \mathbb{S}^{m-1}\subseteq \mathbb{R}^m$. If we consider $axa$, we may decompose
$$x=x_{a\parallel}+x_{a\perp},$$
where $x_{a\parallel}$ is the projection of $x$ onto $a$ and $x_{a\perp}$ is the remainder part of $x$ perpendicular to $a$. Hence $x_{a\parallel}$ is a scalar multiple of $a$ and we have
$$axa=ax_{a\parallel}a+ax_{a\perp}a=-x_{a\parallel}+x_{a\perp}.$$
So the action $axa$ describes a reflection of $x$ in the direction of $a$. By the Cartan-Dieudonn$\acute{e}$ Theorem each $O\in O(m)$ is the composition of a finite number of reflections. If $a=y_1\cdots y_p\in Pin(m),$ we define $\tilde{a}:=y_p\cdots y_1$ and observe $ax\tilde{a}=O_a(x)$ for some $O_a\in O(m)$. Choosing $y_1,\ \dots,\ y_p$ arbitrarily in $\mathbb{S}^{m-1}$, we have the group homomorphism
\begin{eqnarray*}
\theta:\ Pin(m)\longrightarrow O(m)\ :\ a\mapsto O_a,
\end{eqnarray*}
with $a=y_1\cdots y_p$ and $O_ax=ax\tilde{a}$ is surjective. Further $-ax(-\tilde{a})=ax\tilde{a}$, so $1,\ -1\in Ker(\theta)$. In fact $Ker(\theta)=\{1,\ -1\}$. See \cite{P1}. The Spin group is defined as
$$Spin(m)=\{a\in \mathcal{C}l_m: a=y_1y_2\dots y_{2p},\ y_1,\dots,y_{2p}\in\mathbb{S}^{m-1},p\in\mathbb{N}\}$$
 and it is a subgroup of $Pin(m)$. There is a group homomorphism
\begin{eqnarray*}
\theta:\ Spin(m)\longrightarrow SO(m)
\end{eqnarray*}
that is surjective with kernel $\{1,\ -1\}$ and defined by the above group homomorphism for $Pin(m)$. Thus $Spin(m)$ is the double cover of $SO(m)$. See \cite{P1} for more details.\\
\par
For a domain $U$ in $\Rm$, a diffeomorphism $\phi: U\longrightarrow \mathbb{R}^m$ is said to be conformal if, for each $x\in U$ and each $\mathbf{u,v}\in TU_x$, the angle between $\mathbf{u}$ and $\mathbf{v}$ is preserved under the corresponding differential at $x$, $d\phi_x$.
For $m\geq 3$, a theorem of Liouville tells us the only conformal transformations are M\"obius transformations. Ahlfors and Vahlen show any M\"{o}bius transformation on $\mathbb{R}^m \cup \{\infty\}$ can be expressed as $y=(ax+b)(cx+d)^{-1}$ with $a,\ b,\ c,\ d\in \mathcal{C}l_m$ satisfying the following conditions \cite{Ah}:
\begin{eqnarray*}
&&1.\ a,\ b,\ c,\ d\ are\ all\ products\ of\ vectors\ in\ \mathbb{R}^m.\\
&&2.\ a\tilde{b},\ c\tilde{d},\ \tilde{b}c,\ \tilde{d}a\in\mathbb{R}^m.\\
&&3.\ a\tilde{d}-b\tilde{c}=\pm 1.
\end{eqnarray*}
 Since $y=(ax+b)(cx+d)^{-1}=ac^{-1}+(b-ac^{-1}d)(cx+d)^{-1}$, a conformal transformation can be decomposed as compositions of translation, dilation, reflection and inversion. This gives an \emph{Iwasawa decomposition} for M\"obius transformations. See \cite{Li} for more details.
\par
The Dirac operator in $\mathbb{R}^m$ is defined to be $$D_x:=\sum_{i=1}^{m}e_i\partial_{x_i}.$$  Note $D_x^2=-\Delta_x$, where $\Delta_x$ is the Laplacian in $\mathbb{R}^m$.  A $\Clm$-valued function $f(x)$ defined on a domain $U$ in $\Rm$ is left monogenic if $D_xf(x)=0.$ Since Clifford multiplication is not commutative in general, there is a similar definition for right monogenic functions. Sometimes, we will consider the Dirac operator $D_u$ in a vector $u$ rather than $x$.\\
\par
In classical Clifford analysis, the $k$th order conformally invariant differential operator is $D_x^k$ and a large number of results have been found, for instance, \cite{Fe,P,R,R1}. In particular, the fundamental solutions and the intertwining operators for $D_x^k$ are as follows.
\begin{proposition}\cite{P,R}\textbf{(Fundamental solutions for $D_x^k$)}\\
Let $x\in\Rm$, the fundamental solutions $G_{k}(x)$ for $D_x^k$ are as follows. When $m$ is odd,
\begin{equation*}
G_k(x):=
\begin{cases}
c_{2n}||x||^{2n-m},  &\text{if $k=2n$, $n=1,2,\cdots$,}\\
c_{2n-1}\displaystyle\frac{x}{||x||^{m-2n+2}},   &\text{if $k=2n-1$, $n=1,2,\cdots.$}
\end{cases}
\end{equation*}
When $m$ is even,
\begin{equation*}
G_k(x):=
\begin{cases}
\displaystyle\frac{1}{||x||^{m-2n}},  &\text{if $k=2n$, $n=1,2,\cdots,\frac{m}{2}-1$,}\\
\displaystyle\frac{x}{||x||^{m-2n+2}},   &\text{if $k=2n-1$, $n=1,2,\cdots,\frac{m}{2}-1$.}
\end{cases}
\end{equation*}
\end{proposition}
\begin{proposition}\cite{P,R} \textbf{(Intertwining operators for $D_x^k$)}\label{IODk}\\
Let $y=\varphi(x)=(ax+b)(cx+d)^{-1}$ be a M\"{o}bius transformation. Then we have
\begin{eqnarray*}
J_{-k}(\varphi,x)D^k_yf(y)=D^k_xJ_k(\varphi,x)f((ax+b)(cx+d)^{-1}),
\end{eqnarray*}
where 
\begin{eqnarray*}
&&J_k(\varphi,x)=\frac{\widetilde{cx+d}}{||cx+d||^{m-2j+2}},\ if\ k=2j-1,\\
&& J_k(\varphi,x)=||cy+d||^{2j-m},\ if\ k=2j;\\
&&J_{-k}(\varphi,x)=\frac{cx+d}{||cx+d||^{m+2j}},\ if\ k=2j-1,\\
&& J_{-k}(\varphi,x)=||cx+d||^{-m-2j},\ if\ k=2j,
\end{eqnarray*}
and $j$ is a positive integer. $J_{-k},\ J_{k}$ are called the intertwining operators  and $J_k$ is called the conformal weight for $D_x^k$.
\end{proposition}
In this paper, we will generalize $k$th ($k>2$) order conformally invariant differential operators from classical Clifford analysis to higher spin theory as well as their fundamental solutions and intertwining operators. We start with introducing two well known polynomial spaces and the first and second order conformally invariant differential operators in higher spin theory as follows.\\
\par
Let $\mathcal{M}_k$ denote the space of $\mathcal{C}l_m$-valued monogenic polynomials homogeneous of degree $k$. Note that if $h_k\in\Hk$, the space of $\mathcal{C}l_m$-valued harmonic polynomials homogeneous of degree $k$, then $D_uh_k\in\mathcal{M}_{k-1}$, but $D_uup_{k-1}(u)=(-m-2k+2)p_{k-1}(u),$ so
$$\mathcal{H}_k=\mathcal{M}_k\oplus u\mathcal{M}_{k-1},\ h_j=p_k+up_{k-1}.$$
This is an \emph{Almansi-Fischer decomposition} of $\Hk$ \cite{D}. In this Almansi-Fischer decomposition, we define $P_k$ as the projection map 
\begin{eqnarray*}
P_k: \mathcal{H}_k\longrightarrow \mathcal{M}_k.
\end{eqnarray*}
Suppose $U$ is a domain in $\mathbb{R}^m$. Consider a differentiable function $f: U\times \mathbb{R}^m\longrightarrow \mathcal{C}l_m$
such that, for each $x\in U$, $f(x,u)$ is a left monogenic polynomial homogeneous of degree $k$ in $u$. Then the first order conformally invariant differential operator in higher spin theory, named as Rarita-Schwinger operator \cite{B,D}, is defined by 
\begin{eqnarray}\label{RSoperator}
R_kf(x,u):=P_kD_xf(x,u)=(\frac{uD_u}{m+2k-2}+1)D_xf(x,u).
\end{eqnarray}
Let $Z_k(u,v)$ be the reproducing kernel for $\Mk$, which satisfies
\begin{eqnarray*}
f(v)=\int_{\Sm}\overline{Z_k(u,v)}f(u)dS(u),\ for\ all\ f(v)\in\Mk.
\end{eqnarray*}
Then the fundamental solution for $R_k$ is 
\begin{eqnarray*}
E_{k,1}(x,u,v)=\frac{1}{\omega_{m-1}c_{k,1}}\frac{x}{||x||^m}Z_k(\frac{xux}{||x||^2},v),
\end{eqnarray*}
where constant $c_{k,1}$ is $\displaystyle\frac{m-2}{m+2k-2}$ and $\omega_{m-1}$ is the area of $(m-1)$-dimensional unit sphere.\\
\par
In other words, $R_k$ can be considered as the inverse of $E_{k,1}(x,u,v)$ in the following sense.
\begin{proposition}
For any $\phi(y,v)\in C^{\infty}(\Rm,\Mk)$ with compact support with respect to variable $x$, we have
\begin{eqnarray*}
\iint_{\Rm}(R_kE_{k,1}(x-y,u,v),\phi(x,v))_vdx^m=\phi (y,u).
\end{eqnarray*}
where
\begin{eqnarray*}
(f(v),g(v))_v=\int_{\Sm}f(v)g(v)dS(v)
\end{eqnarray*}
is the Fischer-inner product for two Clifford valued polynomials.
\end{proposition}
\par
The second order conformally invariant differential operator in higher spin theory, named the higher spin Laplace operator \cite{B1}, is defined by
 $$\Dtwo=\Delta_x-\displaystyle\frac{4\udx\dudx}{m+2k-2}+\displaystyle\frac{||u||^2\dudx^2}{(m+2k-2)(m+2k-4)}.$$ 
 Its fundamental solution is given by
 $$E_{k,2}(x,u,v)=\displaystyle\frac{(m+2k-4)\Gamma(\displaystyle\frac{m}{2}-1)}{4(4-m)\pi^{\frac{m}{2}}}||x||^{2-m}Z_k(\displaystyle\frac{xux}{||x||^2},v),$$ 
where $Z_k(u,v)$ is the reproducing kernel for $\Hk$ and satisfies
\begin{eqnarray*}
f(v)=\int_{\Sm}\overline{Z_k(u,v)}f(u)dS(u),\ for\ all\ f(v)\in\Hk.
\end{eqnarray*} 
Also $\Dtwo$ can also considered as the inverse of $E_{k,2}(x,u,v)$ in a similar sense as for $R_k$. This will be studied in a more general setting in Section $4.1$.\\
\par
Though we have presented the Almansi-Fischer decomposition, the Dirac operator, and the Rarita-Schwinger operator here in terms of functions taking values in the real Clifford algebra $\mathcal{C}l_m$, they can all be realized in the same way for spinor-valued functions in the complex Clifford algebra $\Clm (\C)$, see \cite{De}; we discuss spinors in the next section.

%%%%%%%%%%%%%%%%%%%%        Irreducible representations of Spin group            %%%%%%%%%%%%%%%%%%%%
\subsection{Irreducible representations of the Spin group}
We now introduce three representations of $Spin(m)$. The first representation of the Spin group is used as the target space in spinor-valued theory and the other two representations of the Spin group are frequently used as target spaces in higher spin theory.
%%%%%%%%%%%%%%
\subsubsection{Spinor representation space $\mathcal{S}$}
The most commonly used representation of the Spin group in $\mathcal{C}l_m(\C)$-valued function theory is the spinor space. To this end, consider the complex Clifford algebra $\mathcal{C}l_m(\mathbb{C})$ with even dimension $m=2n$. The space of vectors $\C^m$ is embedded in $\Clm(\C)$ as
\begin{eqnarray*}
(x_1,x_2,\cdots,x_m)\mapsto \sum^{m}_{j=1}x_je_j:\ \C ^m\hookrightarrow \mathcal{C}l_m(\C).
\end{eqnarray*}
We denote $x$ for a \emph{vector} in both interpretations.
% The space of \emph{$k-$vectors} is defined as
%$$\mathcal{C}l_m(\C)^{(k)}=span_{\mathbb{C}}\{e_{i_1}\dots e_{i_k}:\ 1\leq i_1<\dots<i_k\leq m\}.$$
%The Clifford algebra $\mathcal{C}l_m$ can be rewritten as a direct sum of the even subalgebra and the odd subalgebra:
%$$\mathcal{C}l_m(\C)=\mathcal{C}l_m(\C)^{+}\oplus\mathcal{C}l_m(\C)^{-},$$
%where 
%\begin{eqnarray*}
%\mathcal{C}l_m(\C)^{+}=\oplus_{j=1}^m\mathcal{C}l_m(\C)^{(2j)};\ \mathcal{C}l_m(\C)^{-}=\oplus_{j=1}^m\mathcal{C}l_m(\C)^{(2j-1)}.
%\end{eqnarray*}
The \emph{Witt basis} elements of $\C^{m}$ are defined by 
$$f_j:=\displaystyle\frac{e_{2j-1}-ie_{2j}}{2},\ \ f_j^{\dagger}:=-\displaystyle\frac{e_{2j-1}+ie_{2j}}{2},\ 1\leq j\leq n.$$
Let $I:=f_1f_1^{\dagger}\dots f_nf_n^{\dagger}$. The space of \emph{Dirac spinors} is defined as
$$\mathcal{S}:=\Clm(\C)I.$$ This is a representation of $Spin(m)$ under the following action
$$\rho(s)\scs:=s\scs,\ for\ s\in Spin(m).$$
Note $\scs$ is a left ideal of $\Clm (\C)$. For more details, see \cite{De}. An alternative construction of spinor spaces is given in the classic paper of Atiyah, Bott and Shapiro \cite{At}.

%%%%%%%%       Harmonic polynomials      %%%%%%%%%%%%%
\subsubsection{Homogeneous harmonic polynomials on $\mathcal{H}_k(\Rm,\mathbb{C})$}
It is well known the space of harmonic polynomials is invariant under action of $Spin(m)$, since the Laplacian $\Delta_m$ is an $SO(m)$ invariant operator. It is not irreducible for $Spin(m)$, however, and can be decomposed into the infinite sum of $k$-homogeneous harmonic polynomials, $0\leq k<\infty$. Each of these spaces is irreducible for $Spin(m)$. This brings the most familiar representations of $Spin(m)$: spaces of complex-valued $k$-homogeneous harmonic polynomials on $\mathbb{R}^m$, denoted by $\Hk:=\mathcal{H}_k(\Rm,\mathbb{C})$. Since $\mathbb{C}$ is considered as a scalar subspace of $\mathcal{C}l_m(\mathbb{C})$, $\Hk$ is also called a scalar-valued $k$-homogeneous harmonic polynomial spaces. The following action has been shown to be an irreducible representation of $Spin(m)$ \cite{G,L}: 
%\begin{eqnarray*}
%\rho\ :\ Spin(m)\longrightarrow Aut(\Hk),\ s\longmapsto f(x)\mapsto f(sy\tilde{s})
%\end{eqnarray*}
%with $x=sy\tilde{s}$. This can also be realized as follows
%\begin{eqnarray*}
%Spin(m)\xlongrightarrow{\theta}SO(m)\xlongrightarrow{\rho} Aut(\Hk),\text{  } a\longmapsto O_a\longmapsto \big(f(x)\mapsto f(O_ay)\big),
%\end{eqnarray*}
%where $x=O_ay$, $\theta$ is the double covering map, and $\rho$ is the standard action of $SO(m)$ on a function $f(x)\in\Hk$ with $x\in\mathbb{R}^m$. The function $\phi(z)=(z_1+iz_m)^k$ is the highest weight vector for $\Hk (\Rm,\C)$ having highest weight $(k,0,\cdots,0)$. See \cite{G} for details. Accordingly, we say the spin representations given by $\mathcal{H}_k(\Rm,\mathbb{C})$ have integer spin $k$; we can either specify an integer spin $k$ or the degree of homogeneity $k$ of harmonic polynomials.

\begin{eqnarray*}
\rho\ :\ Spin(m)\longrightarrow Aut(\Hk),\ s\longmapsto (f(x)\mapsto \tilde{s}f(sx\tilde{s})s).
\end{eqnarray*}
This can also be realized as follows
\begin{eqnarray*}
Spin(m)\xlongrightarrow{\theta}SO(m)\xlongrightarrow{\rho} Aut(\Hk);\\
a\longmapsto O_a\longmapsto \big(f(x)\mapsto f(O_ax)\big),
\end{eqnarray*}
where $\theta$ is the double covering map and $\rho$ is the standard action of $SO(m)$ on a function $f(x)\in\Hk$ with $x\in\mathbb{R}^m$. 
%The function $\phi(z)=(z_1+iz_m)^k$ is the highest weight vector for $\Hk (\Rm,\C)$ having highest weight $(k,0,\cdots,0)$ (for more details, see \cite{G}). Accordingly, the spin representations given by $\mathcal{H}_k(\Rm,\mathbb{C})$ are said to have integer spin $k$; we can either specify an integer spin $k$ or the degree of homogeneity $k$ of harmonic polynomials.

%%%%%%%%%%           Monogenic polynomials        %%%%%%%%%%%%%%
\subsubsection{Homogeneous monogenic polynomials on $\mathcal{C}l_m$}
In $\mathcal{C}l_m$-valued function theory, the previously mentioned Almansi-Fischer decomposition shows we can also decompose the space of $k$-homogeneous harmonic polynomials:
$$\Hk=\Mk\oplus u\Mkk.$$
If we restrict $\Mk$ to the spinor valued subspace, we have another important representation of $Spin(m)$: the space of $k$-homogeneous spinor-valued monogenic polynomials on $\Rm$, henceforth denoted by $\Mk:=\Mk (\mathbb{R},\mathcal{S})$. Specifically, the following action has been shown to be an irreducible representation of $Spin(m)$ \cite{G,L}:
\begin{eqnarray*}
\pi\ :\ Spin(m)\longrightarrow Aut(\Mk),\ s\longmapsto f(x)\mapsto \tilde{s}f(sx\tilde{s}).
\end{eqnarray*}
%When $m$ is odd, in terms of complex variables $z_s=x_{2s-1}+ix_{2s}$ for all $1\leq s\leq \frac{m-1}{2}$, the highest weight vector is
%$\omega_k(x)=(\bar{z_1})^kI$ for $\Mk(\Rm,\mathcal{S})$ having highest weight $(k+\frac{1}{2},\frac{1}{2},\cdots,\frac{1}{2})$, where $\bar{z_1}$ is the conjugate of $z_1$; $\mathcal{S}$ is the Dirac spinor space; and $I$ is defined as in \emph{Section 2.2.1}. For details, see \cite{L}. Accordingly, the spin representations given by $\mathcal{M}_k(\Rm,\mathcal{S})$ are said to have half-integer spin $k+\frac{1}{2}$; we can either specify a half-integer spin $k+\frac{1}{2}$ or the degree of homogeneity $k$ of monogenic polynomials.

%%%%%%%%%%     Stein Weiss type operator  %%%%%%%%%%%%%%%%%
\section{Stein-Weiss type operators}\hspace*{\fill} \\
In classical Clifford analysis, the Euclidean Dirac operator was initially motivated from Stokes' Theorem \cite{Johnnote} and Clifford algebras were used to study it. When we consider function theory in higher spin spaces, since these functions take values in irreducible representations of the Spin group, it turns out representation theory provides a quite different approach for operator theory in higher spin spaces. Abundant results have been found with this approach: for instance, \cite{B1,Br1,DavidE,E}. In 1968, Stein and Weiss \cite{ES} pointed out that many first-order differential operators can be constructed as projections of generalized gradients with the help of representation theory. Fegan \cite{Fe} showed that such operators are conformally invariant with certain conditions. In \cite{Ding0,SR}, the Euclidean Dirac and Rarita-Schwinger operators were constructed as Stein-Weiss type operators. Since this construction generalizes further to representations of principal bundles over oriented Riemannian spin manifolds, by which one constructs the Atiyah-Singer Dirac operator, we argue the Stein and Weiss construction is the natural way to construct other Dirac type operators as in \cite{G,SR}. In this section, we recall the constructions of the Euclidean Dirac and Rarita-Schwinger operators as Stein-Weiss type operators from \cite{Ding0}. Motivated by this representation theoretic approach, we will construct other higher order conformally invariant differential operators in higher spin spaces in the next section.
\\
\par
Assume $U$ is a finite dimensional inner product complex vector space, $V$ is a $m$-dimensional inner product complex vector space. Denote the groups of all automorphisms of $U$ and $V$ by $GL(U)$ and $GL(V)$, respectively. Suppose $\rho_1:\ G\longrightarrow GL(U)$ and $\rho_2:\ G\longrightarrow GL(V)$ are irreducible representations of a compact Lie group $G$. Let $f(x)$ be a differentiable function defined on a domain $\Omega\subset\Rm$ with values in $U$. We wish to define the gradient $\nabla f(x)$ as a function from the same domain $\Omega$ but with values in $U\otimes V$. Suppose that $\{\zeta_{\alpha}\}$ is an orthonormal basis in $U$ and $f(x)=\sum_\alpha f_{\alpha}(x)\zeta_{\alpha}$. Let $\{e_1,\cdots,e_m\}$ be the standard basis of $V$ arising from the identification of $V$ with $\mathbb{C}^m$. Then a basis (over $\mathbb{C}$) of $U\otimes V$ is $\{\zeta_{\alpha}\otimes e_i\}_{\alpha,i}$ and 
\begin{eqnarray*}
\nabla f(x)=\sum_{\alpha,i}\frac{\partial f_{\alpha}(x)}{\partial_{x_i}}\zeta_{\alpha}\otimes e_i.
\end{eqnarray*}
In this paper, we rewrite $\nabla f(x)$ as follows for convenience,
\begin{eqnarray*}
\nabla f(x)=\sum_{i}\frac{\partial f(x)}{\partial_{x_i}} e_i.
\end{eqnarray*}
Since $U\otimes V$ is not necessarily irreducible as a tensor product representation of $G$, we denote by $U[\times]V$ the irreducible subrepresentation of $U\otimes V$ whose representation space has largest dimension. This is known as the Cartan product of $\rho_1$ and $\rho_2$. For more details, see \cite{Ea,ES}. Using the inner products on $U$ and $V$, we can write
$$U\otimes V=(U[\times]V)\oplus(U[\times]V)^{\perp}.$$
If we denote by $E$ and $E^{\perp}$ the orthogonal projections onto $U[\times]V$ and $(U[\times]V)^{\perp}$, respectively, then we define differential operators $D$ and $D^{\perp}$ associated to $\rho_1$ and $\rho_2$ by
$$D=E\nabla\text{ and}\ D^{\perp}=E^{\perp}\nabla.$$
These are named \emph{Stein-Weiss type operators} after \cite{ES}. The importance of this construction is that one can reconstruct many first-order differential operators with it by choosing proper representation spaces $U$ and $V$ for a Lie group $G$, such as the Euclidean Dirac operators \cite{ES,SR} and Rarita-Schwinger operators \cite{G} that we now proceed to discuss.\\
\par
\emph{1. Dirac operators}
\hspace*{\fill} \\
\par
Here we only show the odd dimension case, but the even dimension case is similar.
\begin{theorem}
Let $\rho_1$ be the representation of the spin group given by the standard representation of $SO(m)$ on $\Rm$
$$\rho_1:\ Spin(m)\longrightarrow SO(m)\longrightarrow GL(\Rm)$$
and let $\rho_2$ be the spin representation on the spinor space $\mathcal{S}$. Then the Euclidean Dirac operator is the differential operator given by projecting the gradient onto $(\Rm[\times]\mathcal{S})^{\perp}$ when $m=2n+1$.
\end{theorem}
\emph{Outline proof:} The proof is exactly that appearing in \cite{ES}. Let $\{e_1,\cdots,e_m\}$ be an orthonormal basis of $\Rm$ and $x=(x_1,\cdots,x_m)\in\Rm$. For a function $f(x)$ having values in $\scs$, we must show that the system
$$\sum_{i=1}^{m}e_i\displaystyle\frac{\partial f}{\partial x_i}=0$$
is equivalent to the system
$$D^{\perp}f=E^{\perp}\nabla f=0.$$
We have
$$\Rm\otimes\scs=\Rm[\times]\scs \oplus(\Rm[\times]\scs)^{\perp}$$
and \cite{ES} provides an embedding map
\begin{eqnarray*}
&&\eta: \scs\hookrightarrow \Rm\otimes\scs,\\
&& \omega\mapsto \frac{1}{\sqrt{m}}(e_1\omega,\cdots,e_m\omega).
\end{eqnarray*}
Indeed, this embedding is an isomorphism from $\scs$ into $\Rm\otimes\scs$. For the proof, we refer the reader to \emph{page 175} of \cite{ES}. Thus, we have
$$\Rm\otimes\scs=\Rm[\times]\scs \oplus \eta(\scs).$$
Consider the equation $D^{\perp}f=E^{\perp}\nabla f=0$, where $f$ has values in $\scs$. So $\nabla f$ has values in $\Rm\otimes \scs$, and the condition $D^{\perp}f=0$ is equivalent to $\nabla f$ being orthogonal to $\eta (\scs)$. This is precisely the statement that
$$\sum_{i=1}^{m}(\frac{\partial f}{\partial x_i},e_i\omega)=0,\ \forall\omega\in\scs.$$
Notice, however, that as an endomorphism of $\Rm\otimes\scs$, we have $-e_i$ as the dual of $e_i$. Hence the equation above becomes
$$\sum_{i=1}^m(e_i\frac{\partial f}{\partial x_i},\omega)=0,\ \forall\omega\in\scs,$$
which says precisely that $f$ must be in the kernel of the Euclidean Dirac operator. This completes the proof.\qquad \qquad \qquad \qquad \qquad \qquad \qquad \qquad \qquad \qquad\qquad\qquad\qquad \qquad\quad \qedsymbol\\
\par
\emph{2. Rarita-Schwinger operators}
\hspace*{\fill} \\
\par
\begin{theorem}
Let $\rho_1$ be defined as above and $\rho_2$ is the representation of $Spin(m)$ on $\Mk$. Then as a representation of $Spin(m)$, we have the following decomposition
\begin{eqnarray*}
\Mk\otimes \Rm\cong \Mk[\times]\Rm\oplus\Mk\oplus\Mkk\oplus\mathcal{M}_{k,1},
\end{eqnarray*}
where $\mathcal{M}_{k,1}$ is a simplicial monogenic polynomial space as a $Spin(m)$ representation (see more details in \cite{B1}). The Rarita-Schwinger operator is the differential operator given by projecting the gradient onto the $\Mk$ component.
\end{theorem}
\begin{proof}
Consider $f(x,u)\in C^{\infty}(\Rm,\mathcal{M}_k)$. We observe that the gradient of $f(x,u)$ satisfies
$$\nabla f(x,u)=(\partial_{x_1},\cdots,\partial_{x_m})f(x,u)=(\partial_{x_1}f(x,u),\cdots,\partial_{x_m}f(x,u))\in \mathcal{M}_k\otimes \Rm.$$
A similar argument as in \emph {page 181} of \cite{ES} shows 
$$\mathcal{M}_k\otimes \Rm=\Mk[\times]\Rm\oplus V_1\oplus V_2 \oplus V_3,$$
where $V_1\cong \Mk$, $V_2\cong \Mkk$ and $V_3\cong \mathcal{M}_{k,1}$ as $Spin(m)$ representations. Similar arguments as on \emph{page 175} of \cite{ES} show
\begin{eqnarray*}
\theta:\ \Mk\longrightarrow \Mk\otimes\Rm,\text{  }q_k(u)\mapsto (q_k(u)e_1,\cdots,q_k(u)e_m)
\end{eqnarray*} 
is an isomorphism from $\Mk$ into $\Mk\otimes\Rm$. Hence, we have
\begin{eqnarray*}
\mathcal{M}_k \otimes \Rm=\Mk[\times]\Rm\oplus\ \theta(\Mk)\oplus V_2\oplus V_3.
\end{eqnarray*}
Let $P'_k$ be the projection map from $\Mk \otimes \Rm$ to $\theta(\Mk)$. Consider the equation $P'_k\nabla f(x,u)=0$ for $f(x,u)\in C^{\infty}(\Rm,\Mk)$. Then, for each fixed $x$, $\nabla f(x,u)\in\Mk\otimes\Rm$ and the condition $P'_k\nabla f(x,u)=0$ is equivalent to $\nabla f$ being orthogonal to $\theta(\Mk)$. This says precisely
$$\sum_{i=1}^{m}(q_k(u)e_i,\partial_{x_i}f(x,u))_u=0,\ \forall q_k(u)\in\Mk,$$
where $(p(u),q(u))_u=\displaystyle\int_{\Sm}\overline{p(u)}q(u)dS(u)$ is the Fischer inner product for any pair of $\Clm$-valued polynomials. Since $-e_i$ is the dual of $e_i$ as an endomorphism of $\Mk\otimes\Rm$, the previous equation becomes
\begin{eqnarray*}
\sum_{i=1}^m(q_k(u),e_i\partial_{x_i}f(x,u))=(q_k(u),D_xf(x,u))_u=0.
\end{eqnarray*}
Since $f(x,u)\in\Mk$ for fixed $x$, then $D_xf(x,u)\in\Hk$. According to the Almansi-Fischer decomposition, we have
$$D_xf(x,u)=f_1(x,u)+uf_2(x,u), \text{  } f_1(x,u)\in\Mk \text{ and } f_2(x,u)\in\Mkk.$$
We then obtain $(q_k(u),f_1(x,u))_u+(q_k(u),uf_2(x,u))_u=0.$
However, the Clifford-Cauchy theorem \cite{D} shows
$(q_k(u),uf_2(x,u))_u=0.$
Thus, the equation $P'_k\nabla f(x,u)=0$ is equivalent  to $$(q_k(u),f_1(x,u))_u=0,\ \forall q_k(u)\in\Mk.$$
Hence, $f_1(x,u)=0$. We also know, from the construction of the Rarita-Schwinger operator (see (\ref{RSoperator})), that $f_1(x,u)=R_kf(x,u)$. Therefore, the Stein-Weiss type operator $P'_k\nabla$ is precisely the Rarita-Schwinger operator in this context.
\end{proof}
We have demonstrated one application of the Representation-Theoretic approach to Clifford analysis: the Stein-Weiss generalized gradient construction for the Euclidean Dirac and Rarita-Schwinger operators. The operators are realized on irreducible representations of the Spin group. In higher spin theory, we consider operators on functions taking values in irreducible spin representations that have higher spin, i.e., $\Hk$ or $\Mk$. Seeing our success already, we now use the Representation-Theoretic approach to extend the higher spin theory to arbitrary order conformally invariant differential operators of arbitrary spin.

%%%%%%%%%%%%         Higher order higher spin operator  %%%%
\section{Construction and conformal invariance}
Denote the arbitrary $t$-th-order conformally invariant differential operator
\begin{eqnarray*}
\mathcal{D}_t:\ C^{\infty}(\Rm,V)\longrightarrow C^{\infty}(\Rm,V),
\end{eqnarray*}
where the target space $V$ is $\Mk$ or $\Hk$. Thanks to results in \cite{J,Vlad}, the existence and uniqueness (up to a multiplicative constant) of $\mathcal{D}_t$ are already established. More specifically, even order conformally invariant differential operators only exist when $V=\Hk$ and odd order conformally invariant differential operators only exist when $V=\Mk$. This can be easily obtained by taking $\Mk$ or $\Hk$ as the irreducible representation of $Spin(m)$ in Theorems $2$ and $3$ in \cite{Vlad}; these theorems also give the conformal weights of $\mathcal{D}_t$, which provide the intertwining operators of $\mathcal{D}_t$. More specifically, the following result can be obtained from \cite{Vlad}.
\begin{proposition}\label{Intertwiningoperators}
Suppose $y\in\Rm$, $y'=(ay+b)(cy+d)^{-1}$ is a M\"{o}bius transformation and $u'=\displaystyle\frac{(cy+d)u\widetilde{(cy+d)}}{||cy+d||^2}.$ Then
\begin{eqnarray}\label{DtIO}
\mathcal{D}_{t,y',u'}=J_{-t}^{-1}(\varphi,y)\mathcal{D}_{t,y,u}J_t(\varphi,y),
\end{eqnarray}
where
\begin{eqnarray*}
&&J_t(\varphi,y)=\frac{\widetilde{cy+d}}{||cy+d||^{m-2j+2}},\ if\ t=2j-1,\\
&& J_t(\varphi,y)=||cy+d||^{2j-m},\ if\ t=2j;\\
&&J_{-t}(\varphi,y)=\frac{cy+d}{||cy+d||^{m+2j}},\ if\ t=2j-1,\\
&& J_{-t}(\varphi,y)=||cy+d||^{-m-2j},\ if\ t=2j,
\end{eqnarray*}
 and $j$ is a positive integer. Here $J_t$ is called the conformal weight for $\mathcal{D}_t$, and $J_t$ and $J_{-t}$ are called the intertwining operators for $\mathcal{D}_t$. 
 \end{proposition}
More details can be found in \cite{Ding1} Section $3.2$. Let $Z_k(u,v)$ be the reproducing kernel of $\Mk$, which satisifies
\begin{eqnarray*}
f(v)=\int_{\Sm}\overline{Z_k(u,v)}f(u)dS(u),\ for\ all\ f(v)\in\Mk.
\end{eqnarray*}
Recall the fundamental solution of the Rarita-Schwinger operator is $c\displaystyle\frac{x}{||x||^m}Z_k(\displaystyle\frac{xux}{||x||^2},v)$, where $c$ is a non-zero constant \cite{B}. We call $\displaystyle\frac{x}{||x||^m}$ the conformal weight factor and $Z_k(\displaystyle\frac{xux}{||x||^2},v)$ the reproducing kernel factor. The fundamental solution of $D_x^k$ is \cite{P}
 $$c_{2j+1}\displaystyle\frac{x}{||x||^{m-2j}},\ if\ k=2j+1,\text{ and}\quad c_{2j}||x||^{2j-m},\ if\ k=2j, $$
where $c_{2j+1}$ and $c_{2j}$ are non-zero constants. However, when dimension $m$ is even, we also require that $k<m$, because for instance, when $m=k=2j$, the only candidate of fundamental solution is a constant. We expect the fundamental solutions of our higher order higher spin conformally invariant differential operators $\mathcal{D}_t$ to factor into two parts: a conformal weight factor and a reproducing kernel factor, behaving as follows.
\begin{enumerate}
\item The conformal weight factor, i.e., $\displaystyle\frac{x}{||x||^{m-2j}}$ or $||x||^{2j-m}$ term, changes with increasing order similar to the conformal weight for powers of the Dirac operator, differing in the even and odd cases. 
\item The reproducing kernel factor, i.e., $Z_k(\displaystyle\frac{xux}{||x||^2},v)$ term, changes with increasing degree of homogeneity of the target polynomial space similar to the Rarita-Schwinger operator, differing according to whether it is the space of harmonic or monogenic polynomials.
\end{enumerate}
Thus we guess candidates for the fundamental solutions as follows.   
			\begin{enumerate}				
				\item	For $\Deven$, $c||x||^{2j-m}Z_k(\displaystyle\frac{xux}{||x||^2},v)$, where $Z_k(u,v)$ is the reproducing kernel of $\Hk$.
				\item	 For $\Dodd$, $c\displaystyle\frac{x}{||x||^{m-2j+2}}Z_k(\displaystyle\frac{xux}{||x||^2},v)$, where $Z_k(u,v)$ is the reproducing kernel of $\Mk$.
			\end{enumerate}
With similar arguments as in \cite{B,Ding1}, we have the following result.
%%%%%%%%%%%%%%%%%%%%%%%%%%%   Fundamental solutions  %%%%%%%%%%%%%%%%%%%%%%%%%%%%%%%%%
\begin{proposition}\label{fundamentalsolutions}
The fundamental solution for $\Deven$ is $c_{2j}||x||^{2j-m}Z_k(\displaystyle\frac{xux}{||x||^2},v)$, where $Z_k(u,v)$ is the reproducing kernel of $\Hk$ and $c_{2j}$ is a non-zero constant. The fundamental solution for $\Dodd$ is $c_{2j-1}\displaystyle\frac{x}{||x||^{m-2j+2}}Z_k(\displaystyle\frac{xux}{||x||^2},v)$, where $Z_k(u,v)$ is the reproducing kernel of $\Mk$ and $c_{2j-1}$ is a non-zero constant.
\end{proposition}
\begin{proof}
We only give the proof for the fundamental solutions of $\Deven$. A similar argument also applies for $\Dodd$.
Let $Z_k(u,v)$ be the reproducing kernel of $\Hk$, which can be considered as the identity of $End(\mathcal{H}_k)$ and satisfies
\begin{eqnarray*}
P_k(v)=(Z_k(u,v),P_k(u))_u=\int_{S^{m-1}} \overline{Z_k(u,v)}P_k(u)dS(u),\ for\ any\ P_k(u)\in\Hk.
\end{eqnarray*}
%where $(\ ,\ )_u$ denotes the Fischer inner product with respect to $u$; we define the Fischer inner product of two functions by the integral of their product over the sphere, consistent with other work in higher spin theory \cite{B,D}. 
A homogeneous $End(\mathcal{H}_k)$-valued $C^{\infty}$-function $x\rightarrow E(x)$ on $\mathbb{R}^m\backslash \{0\}$ satisfying $\Deven E(x)=\delta(x)Z_k(u,v)$ is referred to as a fundamental solution for the operator $\Deven$. We will show that such a fundamental solution has the form $E_{k,2j}(x,u,v)=c_{2j}||x||^{2j-m}Z_k(\displaystyle\frac{xux}{||x||^2},v)$. Since $Z_k(u,v)$ is a trivial solution of $\Deven$, according to the invariance of $\Deven$ under inversion, we obtain a non-trivial solution $\mathcal{D}_{2j}E_{k,2j}(x,u,v)=0$ in $\mathbb{R}^m\backslash \{0\}$, this can be easily verified from Proposition \ref{Intertwiningoperators} when the M\"{o}bius transformation is inversion. Clearly the function $E_{k,2j}(x,u,v)$ is homogeneous of degree $2j-m$ in $x$, so $\mathcal{D}_{2j}E_{k,2j}(x,u,v)$ is homogeneous of degree $-m$ in $x$ and it belongs to $L_1^{loc}(\mathbb{R}^m)$. Because $\delta(x)$ is the only (up to a multiple) distribution homogeneous of degree $-m$ with support at the origin, we have in the sense of distributions:
\begin{eqnarray*}
\mathcal{D}_{2j}E_{k,2j}(x,u,v)=\delta(x)P_k(u,v)
\end{eqnarray*}
for some $P_k(u,v)\in \mathcal{H}_k\otimes \mathcal{H}_k^*$. Then we have
\begin{eqnarray*}
&&\int_{\mathbb{S}^{m-1}}\mathcal{D}_{2j}\overline{E_{k,2j}(x,u,v)}Q_k(v)dS(v)\\
&=&\delta(x)\int_{\mathbb{S}^{m-1}}\overline{P_k(u,v)}Q_k(v)dS(v).
\end{eqnarray*}
Now, for all $Q_k\in\mathcal{H}_k$, we have
\begin{eqnarray}
&&\int_{\mathbb{S}^{m-1}}\mathcal{D}_{2j}\overline{E_{k,2j}(x,u,v)}Q_k(v)dS(v)\nonumber\\
&=&\mathcal{D}_{2j}\int_{\mathbb{S}^{m-1}}c_{2j}||x||^{2j-m}\overline{Z_k(\frac{xux}{||x||^2},v)}Q_k(v)dS(v)\nonumber\\
&=&\mathcal{D}_{2j}\int_{\mathbb{S}^{m-1}}c_{2j}||x||^{2j-m}\overline{Z_k(\frac{xux}{||x||^2},\frac{xv'x}{||x||^2})}Q_k(\frac{xv'x}{||x||^2})dS(v')\label{D2j},
\end{eqnarray}
where in the last line we made a change of variables in the second argument of $Z_k$. Since $Z_k(u,v)$ is invariant under reflection and $\displaystyle\frac{xux}{||x||^2}$ is a reflection of the variable $u$ in the direction of $x$, in other words (\cite{G}),
\begin{eqnarray*}
Z_k(u,v)=\frac{x}{||x||}Z_k(\frac{xux}{||x||^2},\frac{xvx}{||x||^2})\frac{x}{||x||}=-Z_k(\frac{xux}{||x||^2},\frac{xvx}{||x||^2}).
\end{eqnarray*}
The last equation comes from that $Z_k(\displaystyle\frac{xux}{||x||^2},\displaystyle\frac{xvx}{||x||^2})\in\Hk$, which is scalar valued. Hence, we can commute $\displaystyle\frac{xvx}{||x||^2}$ and $Z_k(\displaystyle\frac{xux}{||x||^2},\displaystyle\frac{xvx}{||x||^2})$. Further, $x^2=-||x||^2$.\\
\par 
Therefore, equation (\ref{D2j}) becomes
\begin{eqnarray*}
&&\mathcal{D}_{2j}\int_{\mathbb{S}^{m-1}}-c_{2j}\overline{Z_k(u,v')}||x||^{2j-m}Q_k(\frac{xv'x}{||x||^2})dS(v')\\
&=&-c_{2j}\mathcal{D}_{k,2j}||x||^{2j-m}Q_k(\frac{xux}{||x||^2}).
\end{eqnarray*}
Hence, we obtain
\begin{eqnarray*}
\delta(x)\int_{\mathbb{S}^{m-1}}\overline{P_k(u,v)}Q_k(v)dS(v)
=-c_{2j}\mathcal{D}_{k,2j}||x||^{2j-m}Q_k(\frac{xux}{||x||^2}).
\end{eqnarray*}

As the reproducing kernel $Z_k(u,v)$ is invariant under the $Spin(m)$-representation
$$H:\ f(u,v)\mapsto \tilde{s}f(su\tilde{s},sv\tilde{s})s,$$ the kernel $E_{k,2j}(x,u,v)$ is also $Spin(m)$-invariant:
\begin{eqnarray*}
\tilde{s}E_{k,2j}(sx\tilde{s},su\tilde{s},sv\tilde{s})s=E_{k,2j}(x,u,v).
\end{eqnarray*}

\noindent From this it follows that $P_k(u,v)$ must be also invariant under $H$. Let now $\phi$ be a test function with $\phi(0)=1$. Let $L$ be the action of $Spin(m)$ given by $L: f(u)\mapsto \tilde{s}f(su\tilde{s})s.$
Then
\begin{eqnarray*}
&&\langle \mathcal{D}_{2j}\big(-c_{2j}||x||^{2j-m}L(\frac{x}{||x||})L(s)Q_k(u)\big),\phi(x)\rangle\\
&=&\int_{\mathbb{S}^{m-1}}\overline{P_k(u,v)}L(s)Q_k(v)dS(v)\\
&=&L(s)\int_{\mathbb{S}^{m-1}}\overline{P_k(u,v)}Q_k(v)dS(v)\\
&=&\langle L(s)\big(-\mathcal{D}_{2j}c_{2j}||x||^{2j-m}L(\frac{x}{||x||})Q_j(u)\big),\phi(x)\rangle.
\end{eqnarray*}
In this way we have constructed an element of $End(\mathcal{H}_k)$ commuting with the $L$-representation of $Spin(m)$ that is irreducible; see Section 2.2.2. By Schur's Lemma (\cite{F}) in representation theory, it follows that $P_k(u,v)$ must be the reproducing kernel $Z_k(u,v)$ if we choose $c_{2j}$ properly. Hence
\begin{eqnarray*}
\mathcal{D}_{2j}E_{k,2j}(x,u,v)=\delta(x)Z_k(u,v).
\end{eqnarray*}
\end{proof}

%%%%%%%%%%%%%%%%%%%%%%%%%%%%%%%%%%%%%%%%%%%%%%%%%%%%%%%%%
We initially expect when the dimension $m$ is even, we must restrict order $2j$ or $2j-1$ to be less than $m$, analogous to the powers of the Dirac operator (see Proposition \ref{IODk}). However, the reproducing kernel factor, i.e., the $Z_k(\displaystyle\frac{xux}{||x||^2},v)$ term in the fundamental solutions, renders this restriction on the order unnecessary for even dimensions. After we can show these fundamental solutions are conformally invariant, constructing a conformally invariant differential operator becomes finding an operator which has a particular fundamental solution.\\
\par
We already found the fundamental solutions for $k$th ($k\geq 1$) order conformally invariant differential operators. This provides us a simple way to recover the higher spin Laplace operator up to a multiplicative constant instead of using generalized symmetries as in \cite{B1}. Consider the twistor and dual twistor operators from the same reference:
\begin{eqnarray*}
	&&T_{k,2}=\langle u,D_x\rangle -\frac{||u||^2\langle D_u,D_x\rangle}{m+2k-4}:\ C^{\infty}(\Rm,\mathcal{H}_{k-1})\longrightarrow C^{\infty}(\Rm,\Hk),\\
	&&T_{k,2}^*=\langle D_u,D_x\rangle:\ C^{\infty}(\Rm,\Hk)\longrightarrow C^{\infty}(\Rm,\mathcal{H}_{k-1}).
\end{eqnarray*}
%Any second order operator $C^{\infty}(\Rm,\Hk)\longrightarrow C^{\infty}(\Rm,\mathcal{H}_{k})$ reduces to a linear combination of 
The second order operators $\Delta_x$ and $T_{k,2}T_{k,2}^*$ map from $C^{\infty}(\Rm,\Hk)$ to $C^{\infty}(\Rm,\Hk)$ and do not change the degree of homogeneity of the variable $u$; more details can be found in \cite{B1}. These are scalar-valued as desired, since $\Hk$ is a scalar-valued function space. It is reasonable, then, to guess the second order bosonic operator of spin $k$ (the higher spin Laplace operator) is a linear combination of these two operators. By our earlier arguments, if there is a linear combination of $\Delta_x$ and $T_{k,2}T_{k,2}^*$ that annihilates $c||x||^{2-m}Z_k(\displaystyle\frac{xux}{||x||^2},v)$, where $Z_k(u,v)$ is the reproducing kernel of $\Hk$ and $c$ is a non-zero constant, then that operator is the higher spin Laplace operator up to a multiplicative constant: 
$$\Dtwo=\Delta_x-\frac{4T_{k,2}T_{k,2}^*}{m+2k-2}.$$
\par
In the rest of this paper, we first introduce convolution type operators associated to fundamental solutions, then we point out fundamental solutions are actually the inverses of the corresponding differential operators in the sense of previous type of convolution. Further we show these convolution type operators are conformally invariant. Therefore, operators with such fundamental solutions are also conformally invariant, considering they are the inverses of their fundamental solutions in the sense of convolution. This also brings us a class of conformally invariant convolution type operators; their inverses, when they exist, are conformally invariant pseudo-differential operators. In classical Clifford analysis, such convolution type operators can be recovered as Knapp-Stein intertwining operators with the help of spinor principal series representations of $Spin(m)$, see \cite{CO}. Hence, our convolution type operators should also be recovered as Knapp-Stein intertwining operators with the principal series representations induced by the polynomial representations of $Spin(m)$ defined in Section $2.2$. However, this is not obvious, and it will be investigated in more detail in an upcoming paper. 

Since the even and odd order conformally invariant differential operators have different target spaces, we will show the constructions in even and odd order cases separately. The even order operators, which have integer spin, are named bosonic operators in analogy with bosons in physics, which are particles of integer spin. Correspondingly, the odd order operators, which have half-integer spin, are named fermionic operators after fermions, which are particles of half-integer spin. It is worth pointing out that the non-zero constants in the fundamental solutions of our conformally invariant differential operators are also determined here. This provides the undetermined constants of the fundamental solutions in the lower spin case in \cite{Ding1}.
%Then we explain how the Rarita-Schwinger operator and higher spin Laplace operator can also be derived from the representation-theoretic approach. 
%%%%%%%%%  Conformal invariance of fundamental solutions %%%%%%%%
\subsection{Convolution type operators}
Assume $E_k(x,u,v)$ is the fundamental solution of $\mathcal{D}_k$. Then we define a convolution operator as follows.
\begin{eqnarray*}
\Phi(f)(y,v)=E_{k}(x-y,u,v)\ast f(x,u):=\int_{\Rm}\int_{\Sm}E_k(x-y,u,v)f(x,u)dS(u)dx^m
\end{eqnarray*}
Notice this is not the usual convolution operator, as it has an integral over the unit sphere with respect to variable $u$. It is worth pointing out that these convolution type operators are actually examples of Knapp-Stein intertwining operators, see \cite{CO}. Since $E_k(x,u,v)$ is the fundamental solution of $\mathcal{D}_k$, we have
$$\mathcal{D}_{k,x,u}E_{k}(x-y,u,v)\ast f(x,u):=\int_{\Rm}\int_{\Sm}\mathcal{D}_{k,x,u}E_k(x-y,u,v)f(x,u)dS(u)dx^m=f(y,v),$$
where $f(y,v)\in C^{\infty}(\Rm,U)$ ($U=\Hk\ or\ \Mk$) with compact support in $y$ for each $v\in\Rm$. 
Hence, we have $\mathcal{D}_kE_k=Id$ and $E_k^{-1}=\mathcal{D}_k$ in the sense above. This implies that if we can show our convolution operator $\Phi$ is conformally invariant, then its corresponding differential operator is also conformally invariant by taking its inverse.
\par
Denote
%\begin{itemize}
%\item 
$$E_{2j}(x,u,v)=||x||^{2j-m}Z_k(\displaystyle\frac{xux}{||x||^2},v)\text{ and }E_{2j-1}(x,u,v)=\displaystyle\frac{x}{||x||^{m-2j+2}}Z_k(\frac{xux}{||x||^2},v),$$
where $Z_k(u,v)$ is the reproducing kernel of $\Hk$ in the even case and the reproducing kernel of $\Mk$ in the odd case.

Next we will show the above convolution operator $\Phi$ is conformally invariant under M\"{o}bius transformations. Thanks to the Iwasawa decomposition, it suffices to verify it is conformally invariant under orthogonal transformation, inversion, translation, and dilation. Conformal invariance under translation and dilation is trivial; hence, we only show the orthogonal transformation and inversion cases here. Incidentally, this method of proof is the first method we mentioned in the introduction for constructing conformally invariant operators, expect for a convolution operator rather than a differential operator; such a method was also mentioned, but not used, in \cite{Ding,Ding1}.
%%%%%
\begin{proposition}\textbf{(Orthogonal transformation)}
Suppose $a\in Spin(m)$ and $x\in\Rm$. If $x'=ax\tilde{a}$, $y'=ay\tilde{a}$, $u'=au\tilde{a}$, and $v'=av\tilde{a}$, then
\begin{enumerate}
\item $E_{2j}(x'-y',u',v')\ast f(x',u')=E_{k}(x-y,u,v)\ast f(ax\tilde{a},au\tilde{a})$,
\item $E_{2j-1}(x'-y',u',v')\ast f(x',u')=aE_{2j-1}(x-y,u,v)\tilde{a}\ast f(ax\tilde{a},au\tilde{a})$
\end{enumerate}
\end{proposition} 
\begin{proof}
\emph{Case 1.} Let $f(x,u)\in C^{\infty}(\Rm,\Hk)$. Since the reproducing kernel of $\Hk$ is rotationally invariant, $ax\tilde{a}$ is a rotation of $x$ in the direction of $a$ for $a\in Spin(m),$ and $a\tilde{a}=1$, we have
 \begin{eqnarray*}
 &&E_{2j}(x'-y',u',v')\ast f(x',u')\\
&=&\int_{\Rm}\int_{\Sm}||x'-y'||^{2j-m}Z_{k}(\frac{(x'-y')u'(x'-y')}{||x'-y'||^2},v')f(x',u')dS(u')dx'^m\\
&=&\int_{\Rm}\int_{\Sm}||a(x-y)\tilde{a}||^{2j-m}Z_{k}(\frac{a(x-y)\tilde{a}au\tilde{a}a(x-y)\tilde{a}}{||a(x-y)\tilde{a}||^2},av\tilde{a})f(ax\tilde{a},au\tilde{a})dS(u)dx^m\\
&=&\int_{\Rm}\int_{\Sm}||x-y||^{2j-m}Z_{k}(\frac{a(x-y)u(x-y)\tilde{a}}{||x-y||^2},av\tilde{a})f(ax\tilde{a},au\tilde{a})dS(u)dx^m\\
&=&\int_{\Rm}\int_{\Sm}||x-y||^{2j-m}Z_{k}(\frac{(x-y)u(x-y)}{||x-y||^2},v)f(ax\tilde{a},au\tilde{a})dS(u)dx^m\\
&=&E_{2j}(x-y,u,v)\ast f(ax\tilde{a},au\tilde{a})
 \end{eqnarray*}
 \par
 \emph{Case 2.} Since the reproducing kernel of $\Mk$ has the property
 $$Z_k(u,v)=\tilde{a}Z_k(au\tilde{a},av\tilde{a})a$$
 for $a\in Spin(m)$, similar argument as in \emph{Case 1} gives the result.
\end{proof}
%%%%%%
\begin{proposition}\textbf{(Inversion)}
Suppose $x\in\Rm$. If $x'=x^{-1}=-\displaystyle\frac{x}{||x||^2}$, $y'=y^{-1}=-\displaystyle\frac{y}{||y||^2}$, $u'=\displaystyle\frac{yuy}{||y||^2}$ and $v'=\displaystyle\frac{xvx}{||x||^2}$, then
\begin{enumerate}
\item $E_{2j}(x'-y',u',v')\ast f(x',u')=-||y||^{m-2j}E_{2j}(x-y,u,v)||x||^{-m-2j}\ast f(x^{-1},\displaystyle\frac{yuy}{||y||^2})$,
\item $E_{2j-1}(x'-y',u',v')\ast f(x',u')=-\big(\displaystyle\frac{y}{||y||^{m-2j+2}}\big)^{-1}E_{2j-1}(x-y,u,v)\displaystyle\frac{x}{||x||^{m-2j}}\ast f(x^{-1},\displaystyle\frac{yuy}{||y||^2})$.
\end{enumerate}
\end{proposition} 
\begin{proof}
\emph{Case 1.}  Suppose $f(x,u)\in C^{\infty}(\Rm,\Hk)$. Notice
$$x^{-1}-y^{-1}=-x^{-1}(x-y)y^{-1}=-y^{-1}(x-y)x^{-1}=-\displaystyle\frac{x}{||x||^2}(x-y)\displaystyle\frac{y}{||y||^2}=-\displaystyle\frac{y}{||y||^2}(x-y)\displaystyle\frac{x}{||x||^2}.$$
 Recall that, as the reproducing kernel of $\Hk$, $Z_k(u,v)$ has the property
$$Z_k(u,v)=-Z_k(\displaystyle\frac{xux}{||x||^2},\displaystyle\frac{xvx}{||x||^2})$$ for $x\in\Rm$.
Hence, we have
\begin{eqnarray*}
 &&E_{2j}(x'-y',u',v')\ast f(x',u')\\
&=&\int_{\Rm}\int_{\Sm}||x'-y'||^{2j-m}Z_{k}(\frac{(x'-y')u'(x'-y')}{||x'-y'||^2},v')f(x',u')dS(u')dx'^m\\
&=&\int_{\Rm}\int_{\Sm}||x^{-1}(x-y)y^{-1}||^{2j-m}Z_k(\frac{x(x-y)yu'y(x-y)x}{|y(|x-y)x||^2},v')f(x^{-1},u')j(x^{-1})dS(u')dx^m\\
&=&\int_{\Rm}\int_{\Sm}-||x^{-1}(x-y)y^{-1}||^{2j-m}Z_k(\frac{(x-y)u(x-y)}{||x-y||^2},v)f(x^{-1},\frac{yuy}{||y||^2})j(x^{-1})dS(u)dx^m
\end{eqnarray*}
where $j(x^{-1})=||x||^{-2m}$ is the Jacobian. Hence,
\begin{eqnarray*}
&=&-\displaystyle\int_{\Rm}\displaystyle\int_{\Sm}||y||^{m-2j}||x-y||^{2j-m}Z_k(\displaystyle\frac{(x-y)u(x-y)}{||x-y||^2},v)||x||^{-m-2j}f(x^{-1},\displaystyle\frac{yuy}{||y||^2})dS(u)dx^m\\
&=&-||y||^{m-2j}E_{2j}(x-y,u,v)||x||^{-m-2j}\ast f(x^{-1},\displaystyle\frac{yuy}{||y||^2}).
\end{eqnarray*}
\par
\emph{Case 2.} Recall that, as the reproducing kernel of $\Mk$, $Z_k(u,v)$ has the property
$$Z_k(u,v)=-\displaystyle\frac{x}{||x||}Z_k(\displaystyle\frac{xux}{||x||^2},\displaystyle\frac{xvx}{||x||^2})\displaystyle\frac{x}{||x||}$$ for $x\in\Rm$. Then, by arguments similar to those above, we have
\begin{eqnarray*}
&&E_{2j-1}(x'-y',u',v')\ast f(x',u')\\
&=&\int_{\Rm}\int_{\Sm}\frac{x'-y'}{||x'-y'||^{m-2j+2}}Z_k(\frac{(x'-y')u'(x'-y')}{||x'-y'||^2})f(x',u')dS(u')dx'^m\\
&=&\int_{\Rm}\int_{\Sm}\frac{y^{-1}(x-y)x^{-1}}{||y^{-1}(x-y)x^{-1}||^{m-2j+2}}\cdot Z_k(\frac{x(x-y)yu'y(x-y)x}{|x(x-y)y||^2},v')f(x^{-1},u')j(x^{-1})dS(u')dx^m\\
&=&\int_{\Rm}\int_{\Sm}\big(\frac{y}{||y||^{m-2j+2}}\big)^{-1}\frac{x-y}{||x-y||^{m-2j+2}}(\frac{x}{||x||^{m-2j+2}})^{-1}\\
&&\quad \quad \quad\cdot-\frac{x}{||x||}Z_k(\frac{(x-y)u(x-y)}{||x-y||^2},v)\frac{x}{||x||}f(x^{-1},\frac{yuy}{||y||^2})||x||^{-2m}dS(u)dx^m\\
&=&\int_{\Rm}\int_{\Sm}-\big(\frac{y}{||y||^{m-2j+2}}\big)^{-1}E_{2j-1}(x-y,u,v)\frac{x}{||x||^{m-2j}}f(x^{-1},\frac{yuy}{||y||^2})dS(u)dx^m\\
&=&-\big(\frac{y}{||y||^{m-2j+2}}\big)^{-1}E_{2j-1}(x-y,u,v)\frac{x}{||x||^{m-2j}}\ast f(x^{-1},\frac{yuy}{||y||^2}).
\end{eqnarray*}
\end{proof}
Hence, the intertwining operators for the convolution operators are as follows.
\begin{proposition}
Suppose $x\in\Rm$, $x'=\varphi(x)=(ax+b)(cx+d)^{-1}$ is a M\"{o}bius transformation, $u'=\displaystyle\frac{(cy+d)u\widetilde{(cy+d)}}{||cy+d||^2},$ and $v'=\displaystyle\frac{(cx+d)v\widetilde{(cx+d)}}{||cx+d||^2}$.
%\begin{eqnarray*}
%&&J_k(\varphi,x)=\frac{\widetilde{cx+d}}{||cx+d||^{m-2j+2}},\ if\ k=2j-1,\\
%&& J_k(\varphi,x)=||cx+d||^{2j-m},\ if\ k=2j;\\
%&&J_{-k}(\varphi,x)=\frac{\widetilde{cx+d}}{||cx+d||^{m+2j}},\ if\ k=2j-1,\\
%&& J_{-k}(\varphi,x)=||cx+d||^{-m-2j},\ if\ k=2j.
%\end{eqnarray*}
Then
\begin{eqnarray*}
E_{k}(x'-y',u',v')\ast f(x',u')
=J_k^{-1}(\varphi,y)E_k(x-y,u,v)J_{-k}(\varphi,x) \ast f(\varphi(x),\displaystyle\frac{(cy+d)u\widetilde{(cy+d)}}{||cy+d||^2}).
\end{eqnarray*}
where $J_k(\varphi,x)$ and $J_{-k}(\varphi,x)$ are defined as in Proposition \ref{Intertwiningoperators}.
\end{proposition}
\par
Recall that $\mathcal{D}_k$ is the inverse of its fundamental solution in the sense of convolution. Hence, we obtain the intertwining operators of $\mathcal{D}_k$ as follows.
\begin{proposition}
Suppose $y\in\Rm$, $y'=(ay+b)(cy+d)^{-1}$ is a M\"{o}bius transformation and $u'=\displaystyle\frac{(cy+d)u\widetilde{(cy+d)}}{||cy+d||^2}.$ Then
\begin{eqnarray*}
\mathcal{D}_{k,y',u'}=J_{-k}^{-1}(\varphi,y)\mathcal{D}_{k,y,u}J_k(\varphi,y).
\end{eqnarray*}
%where $J_k$ and $J_{-k}$ are defined as in Proposition \ref{Intertwiningoperators}.
\end{proposition}
It is worth pointing out that for general $\alpha\in\mathbb{R}$, if we denote 
$$E_k^{\alpha,1}(x-y,u,v)=\displaystyle\frac{x}{||x||^\alpha}Z_k(\displaystyle\frac{xux}{||x||^2},v)$$
where $Z_k(u,v)$ is the reproducing kernel of $\Mk$, and 
$$E_k^{\alpha,2}(x-y,u,v)=||x||^\alpha Z_k(\displaystyle\frac{xux}{||x||^2},v)$$
where $Z_k(u,v)$ is the reproducing kernel of $\Hk$, then we can define a class of convolution type operators
\begin{eqnarray*}
\int_{\Rm}\int_{\Sm}E_k^{\alpha,i}(x-y,u,v)f_i(x,u)dS(u)dx^m
\end{eqnarray*}
where $f_i(x,u)\in C^{\infty}(\Rm,U_i)$ with $U_1=\Mk$ and $U_2=\Hk$. More importantly, these convolution type operators are conformally invariant by similar arguments as above and their inverses, when they exist, are conformally invariant pseudo-differential operators.
\subsection{Bosonic operators: even order, integer spin}
With a similar strategy as in the previous section and arguing by induction, we now construct higher order conformally invariant differential operators in higher spin spaces. We start with the even order case. Denote the $2j$-th order bosonic operator by
\begin{eqnarray*}
\Deven:\ C^{\infty}(\Rm,\Hk)\longrightarrow C^{\infty}(\Rm,\Hk).
\end{eqnarray*}
As the generalization of $D_x^{2j}$ in Euclidean space to higher spin spaces, it is conformally invariant and has the following intertwining operators (Proposition \ref{Intertwiningoperators}):
\begin{eqnarray*}
||cx+d||^{2j+m}\mathcal{D}_{2j,y,\omega}f(y,\omega)=\mathcal{D}_{2j,x,u}||cx+d||^{2j-m}f(\phi(x),\frac{(cx+d)u\widetilde{(cx+d)}}{||cx+d||^2}),
\end{eqnarray*}
where $y=\phi(x)=(ax+b)(cx+d)^{-1}$ is a M\"{o}bius transformation and $\omega=\displaystyle\frac{(cx+d)u\widetilde{(cx+d)}}{||cx+d||^2}$.\\
\par
As mentioned above, the uniqueness (up to a multiplicative constant) and existence of $\Deven$ having the above intertwining operators can be justified by Theorems $2$, $3$ and $4$ of \cite{Vlad} and Chapter 8 of \cite{J}, where the irreducible representation of $Spin(m)$ is $\Hk$. \\
%with highest weight $\lambda=(k,0,\cdots,0)$.
\par
%that the fundamental solution of higher spin Laplace operator is $c ||x||^{2-m}Z_k(\displaystyle\frac{xux}{||x||^2},v)$ \cite{B1}, where $Z_k(u,v)$ is the reproducing kernel of degree-$k$ homogeneous harmonic polynomial space, $||x||^{2-m}$ is the conformal weight $J_{2}$ and $c$ is a non-zero real constant. 
We also have shown that (Proposition \ref{fundamentalsolutions})
$$c_{2j}||x||^{2j-m}Z_k(\displaystyle\frac{xux}{||x||^2},v),$$
is the fundamental solution of $\Deven$, where $c$ is a non-zero real constant. Therefore, to find the $2j$-th order conformally invariant differential operator, we need only find a $2j$-th order differential operator whose fundamental solution is $c||x||^{2j-m}Z_k(\displaystyle\frac{xux}{||x||^2},v)$. Here is our first main theorem.

%%%%%%%%%%%%   Main Theorem  %%%%%%%%%

\begin{theorem}\label{theoremeven}
Let $Z_k(u,v)$ be the reproducing kernel of $\Hk$. When $j>1$, the $2j$-th order conformally invariant differential operator on $C^{\infty}(\Rm,\Hk)$ is the $2j$-th bosonic operator
\begin{eqnarray*}
\Deven=\Dtwo\prod_{s=2}^{j}(\Dtwo-\frac{(2s)(2s-2)}{(m+2k-2)(m+2k-4)}\Delta_x)
\end{eqnarray*}
that has the fundamental solution
$$a_{2j}||x||^{2j-m}Z_k(\frac{xux}{||x||^2},v),$$
where
	$$\Dtwo=\Delta_x-\frac{4T_{k,2}T_{k,2}^*}{m+2k-2}$$
				 is the higher spin Laplace operator \cite{B1}, 
	$$T_{k,2}=\langle u,D_x\rangle -\frac{||u||^2\langle D_u,D_x\rangle}{m+2k-4}\text{ and } T_{k,2}^*=\langle D_u,D_x\rangle$$
	are the second order twistor and dual twistor operators, and $a_{2j}$ is a non-zero real constant whose expression is given later in this section.
\end{theorem}

\par

%%%%%%%%%%%%%%%%%%%%%%%%%%%%%%%%%%%%%%%%%%%

To prove the previous theorem, we start with the following proposition.

%%%%%%%%%%%%%%%%%%%%%%%%  Prop. 1 %%%%%%%%%%%%%%%%%%%%%%%%%%
\begin{proposition}\label{Prop1}
For every $H_{k}(u)\in\Hk(\Rm,\C)$, when $\alpha> 2-m$,
\begin{eqnarray*}
	(\Dtwo-\frac{(m+\alpha)(m+\alpha-2)}{(m+2k-2)(m+2k-4)}\Delta_x)||x||^{\alpha}H_k(\frac{xux}{||x||^2})
	=c_{\alpha+m}||x||^{\alpha-2}H_k(\frac{xux}{||x||^2}),
\end{eqnarray*}
in the distribution sense, where
$$c_{\alpha+m}=-(m+\alpha)(m+\alpha-2)\displaystyle\frac{(\alpha-2k)(\alpha-2k-2)+2k(m+2\alpha-2k-4)}{(m+2k-2)(m+2k-4)}.$$
\end{proposition}

\begin{proof}
In order to prove the above proposition with an arbitrary function $H_k(u)\in\Hk$, as stated in \cite{B1}, we can rely on the fact $\Hk$ is an irreducible $Spin(m)$-representation generated by the highest weight vector $\langle u,2\mathfrak{f}_1\rangle^k$. As $\Dtwo$ and $\Delta_x$ are both $Spin(m)$-invariant operators, it suffices to prove the statement for
$$||x||^{\alpha}\langle \frac{xux}{||x||^2},2\mathfrak{f}_1\rangle^k=||x||^{\alpha-2k}\langle xux,2\fone\rangle^k=||x||^{\alpha-2k}\langle u||x||^2-2\langle u,x\rangle x,2\fone\rangle^k.$$
First, we assume $x\neq 0$. On the one hand, we have
\begin{eqnarray*}
&&\Delta_x||x||^{\alpha-2k}\langle xux,2\fone\rangle^k=\Delta_x||x||^{\alpha-2k}\langle u||x||^2-2\langle u,x\rangle x,2\fone\rangle^k\\
&=&\Delta_x(||x||^{\alpha-2k})\langle xux,2\fone\rangle^k+||x||^{\alpha-2k}\Delta_x(\langle xux,2\fone\rangle^k)+\sum_{j=1}^m\partial_{x_j}(||x||^{\alpha-2k})\partial_{x_j}(\langle xux,2\fone\rangle^k).
\end{eqnarray*}
Since
\begin{eqnarray*}
\partial_{x_j}\langle xux,2\fone\rangle^k=\partial_{x_j}\langle u||x||^2-2\langle u,x\rangle x,2\fone\rangle^k
=k\langle xux,2\fone\rangle^{k-1}\langle 2ux_j-2u_jx-2\langle u,x\rangle e_j,2\fone\rangle,
\end{eqnarray*}
and from \cite{B1}
\begin{eqnarray*}
	\Delta_x\langle xux,2\fone\rangle^k=4k(k-1)||u||^2\langle x,2\fone\rangle^2\langle xux,2\fone\rangle^{k-2}+2k(m+2k-4)\langle u,2\fone\rangle\langle xux,2\fone\rangle^{k-1}.
	\end{eqnarray*}
Therefore,
\begin{eqnarray*}
	&&\Delta_x||x||^{\alpha-2k}\langle xux,2\fone\rangle^k\\
	&=&\big[(\alpha-2k)(\alpha-2k-2)+2k(m-2k+2\alpha-4)\big]||x||^{\alpha-2k-2}\langle xux,2\fone\rangle^k\\
&&+4k(m+2k-4)\langle u,x\rangle\langle x,2\fone\rangle||x||^{\alpha-2k-2}\langle xux,2\fone\rangle^{k-1}\\
&&+4k(k-1)||u||^2\langle x,2\fone\rangle^2||x||^{\alpha-2k}\langle xux,2\fone\rangle^{k-2}.
\end{eqnarray*}
On the other hand, we have \cite{B1}
\begin{eqnarray*}
	&&\Dtwo||x||^{\alpha-2k}\langle xux,2\fone\rangle^k=(m+\alpha-2)(\alpha+\frac{4k}{m+2k-2})||x||^{\alpha-2k-2}\langle xux,2\fone\rangle^k\\
	&&+(m+\alpha-2)(m+\alpha)\frac{4k}{m+2k-2}\langle u,x\rangle\langle x,2\fone\rangle||x||^{\alpha-2k-2}\langle xux,2\fone\rangle^{k-1}\\
	&&+\frac{4k(k-1)(m+\alpha)(m+\alpha-2)}{(m+2k-2)(m+2k-4)}||u||^2\langle x,2\fone\rangle^2||x||^{\alpha-2k}\langle xux,2\fone\rangle^{k-2}.	
	\end{eqnarray*}
	Combining the above two equalities completes the proof when $x\neq 0$.
	Next, we consider the singularity of $\phi(x,u)$ at $x=0$. Notice that singularity only occurs in the $||x||^{\alpha}$ part and that
			 $||x||^{\alpha}$ is weak differentiable if $\alpha>-m+1$ with weak derivative $\partial_{x_i}||x||^{\alpha}=\alpha x_i||x||^{\alpha-2}$. Hence, with the assumption that $\alpha> 2-m$, every differentiation in the process above is also correct in the distribution sense. This completes the proof. 
\end{proof}
%%%%%%%%%%%%%%%%%%%%%%%%%%%%%%%%%%%%%%

Now, we can prove the following proposition immediately.

%%%%%%%%%%%%%     Prop. 2  %%%%%%%%%%%%%%%%%%%
\begin{proposition}\label{prop2}
	When integer $j>1$,
\begin{eqnarray*}
	\Deven a_{2j}||x||^{2j-m}H_k(\frac{xux}{||x||^2})=\delta(x)H_k(u),
\end{eqnarray*}
where $H_k(u)\in\Hk(\Rm,\C)$ and $$a_{2j}=\displaystyle\frac{(m+2k-4)\Gamma(\frac{m}{2}-1)}{4(4-m)\pi^{\frac{m}{2}}}\prod_{s=2}^{j}c_{2s}^{-1},$$
$c_{2s}$ defined by Proposition \ref{Prop1} for $\alpha=2s-m$.
\end{proposition}
\begin{proof}
We prove this proposition by induction. First, when $j=2$,
	\begin{eqnarray*}
	&&\mathcal{D}_4a_4||x||^{4-m}H_k(\frac{xux}{||x||^2})\nonumber\\
	&=&(\Dtwo-\displaystyle\frac{8}{(m+2k-2)(m+2k-4)}\Delta_x)\Dtwo a_4||x||^{4-m}H_k(\frac{xux}{||x||^2})\nonumber\\
	&=&\Dtwo(\Dtwo-\displaystyle\frac{8}{(m+2k-2)(m+2k-4)})\Delta_xa_4||x||^{4-m}H_k(\frac{xux}{||x||^2})\nonumber\\
	&=&\Dtwo\frac{(m+2k-4)\Gamma(\frac{m}{2}-1)}{4(4-m)\pi^{\frac{m}{2}}}||x||^{2-m}H_k(\frac{xux}{||x||^2}),
	\end{eqnarray*}
where the last line follows using $\alpha=4-m$ in Proposition \ref{Prop1}. Thanks to Theorem $5.1$ in \cite{B1}, this last equation is equal to $\delta(x)H_k(u)$.\\
\par
Assume when $j=s$ that the proposition is true. Then for $j=s+1$, we have
\begin{eqnarray*}
&&\mathcal{D}_{2s+2}a_{2s+2}||x||^{2s+2-m}H_k(\frac{xux}{||x||^2})\\
&=&(\Dtwo-\displaystyle\frac{2s(2s+2)}{(m+2k-2)(m+2k-4)}\Delta_x)\mathcal{D}_{2s} a_{2s}c_{2s+2}^{-1}||x||^{6-m}H_k(\frac{xux}{||x||^2})\\
&=&\mathcal{D}_{2s}(\Dtwo-\displaystyle\frac{24}{(m+2k-2)(m+2k-4)}\Delta_x)c_{2s+2}^{-1}a_{2s}||x||^{2s+2-m}H_k(\frac{xux}{||x||^2})\\
&=&\mathcal{D}_{2s} a_{2s}||x||^{2s-m}H_k(\frac{xux}{||x||^2})=\delta(x)H_k(u),
\end{eqnarray*}
where the penultimate equality follows using $\alpha=2s+2-m$ in Proposition \ref{Prop1}. This last equation comes from our assumption $j=s$. Therefore, our proposition is proved.
\end{proof}
%%%%%%%%%%%%%%%%%%%%%%%%%%%%%%%%%%

In particular, from the above proposition, we have
\begin{eqnarray*}
	\Deven a_{2j}||x||^{2j-m}Z_k(\frac{xux}{||x||^2},v)=\delta(x)Z_k(u,v),
\end{eqnarray*}
where $Z_k(u,v)$ is the reproducing kernel of $\Hk$. Hence, Theorem \ref{theoremeven} is proved and the even order case is resolved.

%Notice that when $j=2$, $\Deven$ is the fourth order bosonic operator \cite{Ding}.\\
%\par 
%
%\begin{theorem}
%	The $2j$-th order bosonic operator on $C^{\infty}(\Rm,\Hk)$ is
%	\begin{eqnarray*}
%		\Deven=\prod_{s=1}^{j}(\Dtwo-\frac{(2s)(2s-2)}{(m+2k-2)(m+2k-4)}\Delta_x).
%	\end{eqnarray*}	
%\end{theorem}

%%%%%%%%%%%%%%%%%%%%%%%%  fermionic operators  %%%%%%%%%%%%%%%%
\subsection{Fermionic operators: odd order, half-integer spin}
We denote the $(2j-1)$-th fermionic operator $$\Dodd:\ C^{\infty}(\Rm,\Mk)\longrightarrow C^{\infty}(\Rm,\Mk).$$
as the generalization of $D_x^{2j-1}$ in Euclidean space to higher spin spaces. With similar arguments as in bosonic case, it is conformally invariant with the following intertwining operators
\begin{eqnarray*}
	\frac{\widetilde{cx+d}}{||cx+d||^{m+2j}}\mathcal{D}_{2j-1,y,\omega}f(y,\omega)=\mathcal{D}_{2j-1,x,u}\frac{\widetilde{cx+d}}{||cx+d||^{m-2j+2}}f(\phi(x),\frac{(cx+d)u\widetilde{(cx+d)}}{||cx+d||^2}),
\end{eqnarray*}
where $y=\phi(x)=(ax+b)(cx+d)^{-1}$ is a M\"{o}bius transformation and $\omega=\displaystyle\frac{(cx+d)u\widetilde{(cx+d)}}{||cx+d||^2}$. Furthermore,
its fundamental solution is
$$c_{2j-1}\displaystyle\frac{x}{||x||^{m-2j+2}}Z_k(\displaystyle\frac{xux}{||x||^2},v),$$
where $c_{2j-1}$ is a non-zero real constant and $Z_k(u,v)$ is the reproducing kernel of $\Mk$.
Here comes our second main theorem. The proof is left in Section 5.

%%%%%%%%  Main theorem 2  %%%%%%%%%
\begin{theorem} \label{theoremodd}
Let $Z_k(u,v)$ be the reproducing kernel of $\Mk$. When $j>1$, the $(2j-1)$-th order conformally invariant differential operator on $C^{\infty}(\Rm,\Mk)$ is the $(2j-1)$-th order fermionic operator
	\begin{eqnarray*}
	\Dodd=R_k\prod_{s=1}^{j-1}\bigg(-R_k^2+\frac{4s^2T_kT_k^*}{(m+2k-2s-2)(m+2k+2s-2)}\bigg)
	\end{eqnarray*}
	that has the fundamental solution
	$$\lambda_{2s}\frac{x}{||x||^{m-2j+2}}Z_k(\frac{xux}{||x||^2},v),$$
	where
	\begin{eqnarray*}
	T_k=(1+\frac{uD_u}{m+2k-2})D_x\text{ and } T_k^*=\frac{-uD_uD_x}{m+2k-2}
	\end{eqnarray*}
	 are the twistor and dual twistor operators defined in \cite{Ding0} and $\lambda_{2s}$ is a non-zero real constant whose expression is given in Section 5.
\end{theorem}
On a concluding note, recall that a manifold is conformally flat if it has an atlas whose transition functions are M\"{o}bius transformations. Note this does not involve curvature. Using similar arguments as in our previous paper \cite{Ding2}, one can generalize our conformally invariant differential operators to conformally flat spin manifolds in the fermionic case and conformally flat Riemannian manifolds in the bosonic case. %Further, if $M$ is a conformally flat manifold with spin structure then the conformal weight structure allows us to lift the fermionic differential operators to act on sections on certain bundles.
More specifically, from Proposition \ref{Intertwiningoperators}, if we choose a particular M\"{o}bius transformation, we can generalize our fermionic and bosonic operators to some conformally flat manifold, such as the unit sphere, using Equation (\ref{DtIO}). This will be developed more formally elsewhere.

\section{Explicit proof for the construction of $\Dodd$}
To prove Theorem \ref{theoremodd}, we start with the following proposition.

%%%%%%%%%%%  Prop. 3  %%%%%%%%%%%%%%%%
\begin{proposition}\label{prop3}
	For any $f_k(u)\in\Mk$, we denote
	$$B_{m-\beta}=\Delta_x+a_{m-\beta}||u||^2\dudx^2+b_{m-\beta}\udx\dudx+c_{m-\beta}u\dudx D_x.$$
When $\beta\leq m-2$, we have
	\begin{eqnarray*}
	B_{m-\beta} \frac{x}{||x||^{\beta}}f_k(\frac{xux}{||x||^2})
	=d_{m-\beta}\frac{x}{||x||^{\beta+2}}f_k(\frac{xux}{||x||^2})
	\end{eqnarray*}
	in the distribution sense, where 
	\begin{eqnarray*}
	&&a_{m-\beta}=\frac{4}{(\beta+2k-2)(2m+2k-\beta-2)};\\
	&&b_{m-\beta}=-\frac{4(m+2k-2)}{(\beta+2k-2)(2m+2k-\beta-2)};\\
	&&c_{m-\beta}=-\frac{4}{(\beta+2k-2)(2m+2k-\beta-2)};\\
	&&d_{m-\beta}=(\beta+2k)(\beta+2k-m)+2k(m-2\beta-2k-2)+\frac{4k(m+2k-2)}{\beta+2k-2}.
	\end{eqnarray*}
\end{proposition}
%%%%%%%%%%%%%%%%%%%%%%%%%%%%%%%%%%%%%%
It is worth pointing out that if $\beta=m-2s$, then $B_{2s}$ is exactly the term in the parenthesis in Theorem \ref{theoremodd}. Details can be found later in this section.\\
\par
	In order to prove the above proposition with arbitrary function $f_k(u)\in\Mk$, as stated in \cite{B1}, we can rely on the fact $\Mk$ is an irreducible $Spin(m)$-representation generated by the highest weight vector $\langle u,2\mathfrak{f}_1\rangle^kI$, where $I$ is defined in \emph{Section 2.2.1}. It suffices to prove the statement for
	$$\frac{x}{||x||^{\beta}}\langle \frac{xux}{||x||^2},2\mathfrak{f}_1\rangle^kI=\frac{x}{||x||^{\beta+2k}}\langle xux,2\fone\rangle^kI=\frac{x}{||x||^{\beta+2k}}\langle u||x||^2-2\langle u,x\rangle x,2\fone\rangle^kI.$$
	First, we assume that $x\neq 0$, and we have the following technical lemmas.
	
	%%%%%%%%%% Lemma 1  %%%%%%%
	\begin{lemma}\label{lemma1}
		\begin{eqnarray*}
		&&\Delta_x\frac{x}{||x||^{\beta+2k}}\langle xux,2\fone\rangle^kI\\
		&&=\big[(\beta+2k)(\beta+2k-m)+2k(m-2\beta-2k-2)\big]\frac{x}{||x||^{\beta+2k+2}}\langle xux,2\fone\rangle^kI\\
		&&-4k\frac{u\langle x,2\fone\rangle}{||x||^{\beta+2k}}\langle xux,2\fone\rangle^{k-1}I+4k(m+2k-2)\frac{x}{||x||^{\beta+2k+2}}\langle u,x\rangle\langle x,2\fone\rangle\langle xux,2\fone\rangle^{k-1}I\\
		&&+4k(k-1)||u||^2\langle x,2\fone\rangle^2\frac{x}{||x||^{\beta+2k}}\langle xux,2\fone\rangle^{k-2}I.
		\end{eqnarray*}
	\end{lemma}
	\begin{proof}
		Since
		$$\Delta_x\displaystyle\frac{x}{||x||^{\beta+2k}}=(\beta+2k)(\beta+2k-m)\displaystyle\frac{x}{||x||^{\beta+2k+2}}$$
		and \cite{B1} gives
		\begin{eqnarray*}
		\Delta_x\langle xux,2\fone\rangle^k=4k(k-1)||u||^2\langle x,2\fone\rangle^2\langle xux,2\fone\rangle^{k-2}I+2k(m+2k-4)\langle u,2\fone\rangle\langle xux,2\fone\rangle^{k-1}I,
		\end{eqnarray*}
	 we have
		\begin{eqnarray*}
		&&\Delta_x\frac{x}{||x||^{\beta+2k}}\langle xux,2\fone\rangle^{k}I\\
		&=&\Delta_x(\frac{x}{||x||^{\beta+2k}})\langle xux,2\fone\rangle^{k}I+\frac{x}{||x||^{\beta+2k}}\Delta_x(\langle xux,2\fone\rangle^{k})I
		+2\sum_{i=1}^m\partial_{x_i}\frac{x}{||x||^{\beta+2k}}\partial_{x_i}\langle xux,2\fone\rangle^{k}I\\
		&=&(\beta+2k)(\beta+2k-m)\frac{x}{||x||^{\beta+2k+2}}\langle xux,2\fone\rangle^{k}I\\
		&&+\frac{x}{||x||^{\beta+2k}}\big(4k(k-1)||u||^2\langle x,2\fone\rangle^2\langle xux,2\fone\rangle^{k-2}I+2k(m+2k-4)\langle u,2\fone\rangle\langle xux,2\fone\rangle^{k-1}\big)I\\
		&&+2k\sum_{i=1}^m(\frac{e_i}{||x||^{\beta+2k}}-\frac{(\beta+2k)x_ix}{||x||^{\beta+2k+2}})\langle 2ux_i-2u_ix-2\langle u,x\rangle e_i,2\fone\rangle\langle xux,2\fone\rangle^{k-1}I.
		\end{eqnarray*}
		Notice that $I=\fone\mathfrak{f}_1^{\dagger}\mathfrak{f}_2\mathfrak{f}_2^{\dagger}\cdots\mathfrak{f}_n\mathfrak{f}_n^{\dagger}$ and $\fone^2=0$. Therefore, we obtain
		\begin{eqnarray*}
		&&=(\beta+2k)(\beta+2k-m)\frac{x}{||x||^{\beta+2k+2}}\langle xux,2\fone\rangle^kI\\
		&&+4k(k-1)||u||^2\langle x,2\fone\rangle^2\frac{x}{||x||^{\beta+2k}}\langle xux,2\fone\rangle^{k-2}I\\
		&&+2k(m-2\beta-2k-2)\frac{x}{||x||^{\beta+2k}}\langle u,2\fone\rangle\langle xux,2\fone\rangle^{k-1}I\\
		&&-4k\frac{u\langle x,2\fone\rangle}{||x||^{\beta+2k}}\langle xux,2\fone\rangle^{k-1}I+8k(\beta+2k)\frac{x\langle u,x\rangle\langle x,2\fone\rangle}{||x||^{\beta+2k+2}}\langle xux,2\fone\rangle^{k-1}I.
		\end{eqnarray*}
	With the help of $\langle u||x||^2,2\fone\rangle=\langle xux,2\fone\rangle+2\langle u,x\rangle\langle x,2\fone\rangle,$
	this lemma is proved immediately.
	\end{proof}
	%%%%%%%%%%%%%%%%%%%%%%%%%%%%%%%%%%%%%%%%%%%%%%%%%
	
	%%%%%%   Lemma 2  %%%%%%%%%
	\begin{lemma}\label{lemma2}
		\begin{eqnarray}
		&&||u||^2\dudx^2\frac{x}{||x||^{\beta+2k}}\langle xux,2\fone\rangle^kI\nonumber \\
		&=&k(k-1)(2m-\beta+2k-2)(2m-\beta+2k-4)\frac{||u||^2x}{||x||^{\beta+2k}}\langle x,2\fone\rangle^2\langle xux,2\fone\rangle^{k-2}I;\label{lemma2.1}\\
		&&u\dudx D_x\frac{x}{||x||^{\beta+2k}}\langle xux,2\fone\rangle^kI\nonumber \\
		&=&-k(2m-\beta+2k-2)\bigg[(\beta-m)\frac{u\langle x,2\fone\rangle}{||x||^{\beta+2k}}\langle xux,2\fone\rangle^{k-1}I\nonumber\\
		&&+2(k-1)||u||^2\frac{x}{||x||^{\beta+2k}}\langle x,2\fone\rangle^2\langle xux,2\fone\rangle^{k-2}I\bigg]\label{lemma2.2};\\
		&&\udx\dudx\frac{x}{||x||^{\beta+2k}}\langle xux,2\fone\rangle^kI\nonumber \\
			&=&-k(2m-\beta+2k-2)\bigg[\frac{x\langle xux,2\fone\rangle^k}{||x||^{\beta+2k+2}}I-(\beta+2k-2)\frac{x\langle u,x\rangle\langle x,2\fone\rangle}{||x||^{\beta+2k+2}}\langle xux,2\fone\rangle^{k-1}I\nonumber\\
			&&+\frac{u\langle x,2\fone\rangle}{||x||^{\beta+2k}}\langle xux,2\fone\rangle^{k-1}I-2(k-1)\frac{||u||^2x}{||x||^{\beta+2k}}\langle x,2\fone\rangle^2\langle xux,2\fone\rangle^{k-2}I\bigg]\label{lemma2.3}
			\end{eqnarray}
	\end{lemma}
	\begin{proof}
		Since these three operators on the left contain $\dudx$, first let us check:
		\begin{eqnarray*}
		&&\dudx \frac{x}{||x||^{\beta+2k}}\langle xux,2\fone\rangle^kI
		=\sum_{i=1}^m\partial_{u_i}\bigg(\big(\frac{e_i}{||x||^{\beta+2k}}-\frac{(\beta+2k)x_ix}{||x||^{\beta+2k+2}}\big)\langle xux,2\fone\rangle^kI\\
		&&+k\frac{x}{||x||^{\beta+2k}}\langle xux,2\fone\rangle^{k-1}\langle 2ux_i-2u_ix-2\langle u,x\rangle e_i,2\fone\rangle I\bigg)\\
		&=&\sum_{i=1}^mk\big(\frac{e_i}{||x||^{\beta+2k}}-\frac{(\beta+2k)x_ix}{||x||^{\beta+2k+2}}\big)\langle xux,2\fone\rangle^{k-1}\langle e_i||x||^2-2x_ix,2\fone\rangle I\\
		&&+\sum_{i=1}^mk(k-1)\frac{x\langle xux,2\fone\rangle^{k-2}}{||x||^{\beta+2k}}\langle e_i||x||^2-2x_ix,2\fone\rangle\langle 2ux_i-2u_ix-2\langle u,x\rangle e_i,2\fone\rangle I\\
		&&+\sum_{i=1}^mk\frac{x\langle xux,2\fone\rangle^{k-1}}{||x||^{\beta+2k}}\langle 2e_ix_i-2x-2x_ie_i,2\fone\rangle I.
		\end{eqnarray*}
		The last expression simplifies as
		$$-k(2m-\beta+2k-2)\displaystyle\frac{x\langle x,2\fone\rangle}{||x||^{\beta+2k}}\langle xux,2\fone\rangle^{k-1}I.$$
		Hence, to verify Eq. (\ref{lemma2.1}), we only need to check
		\begin{eqnarray*}
		&&\dudx \displaystyle\frac{x\langle x,2\fone\rangle}{||x||^{\beta+2k}}\langle xux,2\fone\rangle^{k-1}I
		=\sum_{i=1}^m\partial_{u_i}\bigg(\big(\frac{e_i}{||x||^{\beta+2k}}-\frac{(\beta+2k)x_ix}{||x||^{\beta+2k+2}}\big)\langle x,2\fone\rangle\langle xux,2\fone\rangle^{k-1}I\\
		&&+\frac{x\langle e_i,2\fone\rangle}{||x||^{\beta+2k}}\langle xux,2\fone\rangle^{k-1}I+(k-1)\frac{x\langle x,2\fone\rangle}{||x||^{\beta+2k}}\langle xux,2\fone\rangle^{k-2}\langle 2ux_i-2u_ix-2\langle u,x\rangle e_i,2\fone\rangle I\bigg)\\
		&=&\sum_{i=1}^m\big(\frac{e_i}{||x||^{\beta+2k}}-\frac{(\beta+2k)x_ix}{||x||^{\beta+2k+2}}\big)\langle x,2\fone\rangle(k-1)\langle xux,2\fone\rangle^{k-2}\langle e_i||x||^2-2x_ix,2\fone\rangle I\\
		&&+\sum_{i=1}^m\frac{x\langle e_i,2\fone\rangle}{||x||^{\beta+2k}}(k-1)\langle xux,2\fone\rangle^{k-2}\langle e_i||x||^2-2x_ix,2\fone\rangle I\\
		&&+\sum_{i=1}^m(k-1)\frac{x\langle x,2\fone\rangle}{||x||^{\beta+2k}}\langle xux,2\fone\rangle^{k-2}\langle 2e_ix_i-2x-2x_i e_i,2\fone\rangle I\\
		&&+\sum_{i=1}^m(k-1)\frac{x\langle x,2\fone\rangle}{||x||^{\beta+2k}}(k-2)\langle xux,2\fone\rangle^{k-3}\langle 2ux_i-2u_ix-2\langle u,x\rangle e_i,2\fone\rangle\langle e_i||x||^2-2x_ix,2\fone\rangle I.
		\end{eqnarray*}
		This last expression simplifies as
		$$-(k-1)(2m-\beta+2k-4)\frac{x\langle x,2\fone\rangle^2}{||x||^{\beta+2k}}\langle xux,2\fone\rangle^{k-2}I.$$
		Hence, Eq. (\ref{lemma2.1}) is verified.\\
		\par
		For Eq. (\ref{lemma2.2}), we check
		\begin{eqnarray*}
		&&uD_x\frac{x\langle x,2\fone\rangle}{||x||^{\beta+2k}}\langle xux,2\fone\rangle^{k-1}I
		=u\sum_{i=1}^me_i\bigg(\big(\frac{e_i}{||x||^{\beta+2k}}-\frac{(\beta+2k)x_ix}{||x||^{\beta+2k+2}}\big)\langle x,2\fone\rangle\langle xux,2\fone\rangle^{k-1}I\\
		&&+\frac{x\langle e_i,2\fone\rangle}{||x||^{\beta+2k}}\langle xux,2\fone\rangle^{k-1}I+(k-1)\frac{x\langle x,2\fone\rangle}{||x||^{\beta+2k}}\langle xux,2\fone\rangle^{k-2}\langle 2ux_i-2u_ix-2\langle u,x\rangle e_i,2\fone\rangle I\bigg)\\
		&=&u\bigg[\frac{(\beta+2k-m)\langle x,2\fone\rangle}{||x||^{\beta+2k}}\langle xux,2\fone\rangle^{k-1}I-2\frac{\langle x,2\fone\rangle}{||x||^{\beta+2k}}\langle xux,2\fone\rangle^{k-1}I\\
		&&-2(k-1)\frac{\langle x,2\fone\rangle\langle u||x||^2,2\fone\rangle}{||x||^{\beta+2k}}\langle xux,2\fone\rangle^{k-2}I\\
		&&-2(k-1)\frac{ux}{||x||^{\beta+2k}}\langle x,2\fone\rangle^2\langle xux,2\fone\rangle^{k-2}I
		+4(k-1)\frac{\langle u,x\rangle\langle x,2\fone\rangle^2}{||x||^{\beta+2k}}\langle xux,2\fone\rangle^{k-2}I\bigg]\\
		&=&(\beta-m)\frac{u\langle x,2\fone\rangle}{||x||^{\beta+2k}}\langle xux,2\fone\rangle^{k-1}I+2(k-1)\frac{||u||^2x}{||x||^{\beta+2k}}\langle x,2\fone\rangle^2\langle xux,2\fone\rangle^{k-2}I.
		\end{eqnarray*}
		For Eq. (\ref{lemma2.3}), we check
		\begin{eqnarray*}
		&&\udx \frac{x\langle x,2\fone\rangle}{||x||^{\beta+2k}}\langle xux,2\fone\rangle^{k-1}I
		=\sum_{i=1}^mu_i\bigg(\big(\frac{e_i}{||x||^{\beta+2k}}-\frac{(\beta+2k)x_ix}{||x||^{\beta+2k+2}}\big)\langle x,2\fone\rangle\langle xux,2\fone\rangle^{k-1}I\\
		&&+\frac{x\langle e_i,2\fone\rangle}{||x||^{\beta+2k}}\langle xux,2\fone\rangle^{k-1}+(k-1)\frac{x\langle x,2\fone\rangle}{||x||^{\beta+2k}}\langle xux,2\fone\rangle^{k-2}\langle 2ux_i-2u_ix-2\langle u,x\rangle e_i,2\fone\rangle I\bigg)\\
		&=&\frac{x\langle xux,2\fone\rangle^{k}}{||x||^{\beta+2k+2}}I-(\beta+2k-2)\frac{x\langle u,x\rangle\langle x,2\fone\rangle}{||x||^{\beta+2k-2}}\langle xux,2\fone\rangle^{k-1}I\\
		&&+\frac{u\langle x,2\fone\rangle}{||x||^{\beta+2k}}\langle xux,2\fone\rangle^{k-1}I
		-2(k-1)\frac{||u||^2x}{||x||^{\beta+2k}}\langle x,2\fone\rangle^2\langle xux,2\fone\rangle^{k-2}I.
		\end{eqnarray*}
		Therefore, Eqs. (\ref{lemma2.2}) and (\ref{lemma2.3}) are verified.
	\end{proof}
	Recall the fact mentioned in the $2j$-th order bosonic operator case,	
	$||x||^{\alpha}$ is weak differentiable if $\alpha>-m+1$ with weak derivative $\partial_{x_i}||x||^{\alpha}=\alpha x_i||x||^{\alpha-2}$. Hence, when $\beta\leq m-2$, Lemmas \ref{lemma1} and \ref{lemma2} are both true in the distribution sense. Combining them completes the proof of Proposition \ref{prop3}. With the help of Proposition \ref{prop3} and similar arguments as in Proposition \ref{prop2}, we have the following proposition by induction.
\begin{proposition}\label{prop4}
	Let $f_k(u)\in\Mk$. When integer $j>1$,
	\begin{eqnarray*}
		\Bigg[\prod_{s=1}^{j-1}B_{2s}d_{2s}^{-1}\Bigg] \frac{x}{||x||^{m-2j+2}}f_k(\frac{xux}{||x||^2})
		=\frac{x}{||x||^{m}}f_k(\frac{xux}{||x||^2})
	\end{eqnarray*}
	in the distribution sense, where $a_{2s},\ b_{2s},\ c_{2s},\ d_{2s}$ are defined as in Proposition \ref{prop3} with $\beta=m-2s$.
\end{proposition}
	In the above proposition, it is worth pointing out that $$B_{2s_1}B_{2s_2}=B_{2s_2}B_{2s_1},$$ 
where $s_1\neq s_2.$ Indeed, with a straightforward calculation, one can get
\begin{eqnarray*}
R_k^2=-\Delta_x+\frac{4\udx\dudx}{m+2k-2}-\frac{4||u||^2\dudx^2}{(m+2k-2)^2}+\frac{4u\dudx D_x}{(m+2k-2)^2}.
\end{eqnarray*}	
%where $R_k=(1+\displaystyle\frac{uD_u}{m+2k-2})D_x$ is the Rarita-Schwinger operator. 
Then
\begin{eqnarray}
&&B_{2s}=\Delta_x+\frac{4\big(||u||^2\dudx^2-(m+2k-2)\udx\dudx^2-u\dudx D_x\big)}{(m+2k-2s-2)(m+2k+2s-2)}\nonumber\\
&=&\Delta_x-\frac{(m+2k-2)^2}{(m+2k-2s-2)(m+2k+2s-2)}(R_k^2+\Delta_x).\label{equation7}
\end{eqnarray}
So $B_{2s}$ is a linear combination of $R_k^2$ and $\Delta_x$. This is no surprise, since \cite{DavidE} points out
\begin{eqnarray*}
\{R_k^i\Delta_x^j,\ 0\leq i\leq min(2p+1,2k+1),\ 0\leq j,\ i+2j=p\}
\end{eqnarray*}	
is the basis of the space of $Spin(m)$-invariant constant coefficient differential operators of order $p$ on $\Mk$. $\Dodd$ is conformally invariant, so it is also $Spin(m)$-invariant and hence can be expressed in this basis. Furthermore, with the help of $-\Delta_x=R_k^2+T_kT_k^*$ and Eq. (\ref{equation7}), we can also rewrite $B_{2s}$ in terms of first order conformally invariant operators:
\begin{eqnarray*}
B_{2s}=-R_k^2+\frac{4s^2T_kT_k^*}{(m+2k-2s-2)(m+2k+2s-2)}.
\end{eqnarray*}

\par

Now, we have fundamental solution of $\Dodd$ restated as follows.
\begin{theorem*}
Let $Z_k(u,v)$ be the reproducing kernel of $\Mk$. When $j>1$, the $(2j-1)$-th order fermionic operator $\Dodd$
	has fundamental solution
	\begin{equation*} \lambda_{2s}\frac{x}{||x||^{m-2j+2}}Z_k(\frac{xux}{||x||^2},v), \text{  } \lambda_{2s}=\frac{-(m+2k-2)}{(m-2)\omega_{m-1}}\prod_{s=1}^{j-1}d_{2s}^{-1},\end{equation*} where $d_{2s}$ is defined in Proposition \ref{prop3} with $\beta=m-2s$ and $\omega_{m-1}$ is the area of $(m-1)$-dimensional unit sphere.
\end{theorem*}
\begin{proof}
	With the help of Proposition \ref{prop4} and noticing that $$\displaystyle\frac{-(m+2k-2)}{(m-2)\omega_{m-1}}\displaystyle\frac{x}{||x||^{m}}Z_k(\frac{xux}{||x||^2},v)$$
	is the fundamental solution of $R_k$ \cite{B}, the above theorem follows immediately.
\end{proof}
Hence, Theorem \ref{theoremodd} is proved and the odd order case is resolved. \\
\par

\end{document}